\documentclass[a4paper,12pt]{article}

\usepackage{xcolor}

\usepackage[margin=3cm,footskip=1cm]{geometry}

\usepackage[T1]{fontenc}

\usepackage{amsmath,amssymb,amsthm,bm,mathtools,xspace,booktabs,url}
\usepackage{graphicx,etoolbox}

\usepackage[round]{natbib}

\AtBeginEnvironment{thebibliography}{%
  \small
  \setlength\itemsep{0.75em plus 0.5em minus 0.75em}%
}

\usepackage{caption}
\usepackage[raggedright]{titlesec}

\numberwithin{equation}{section} 

\theoremstyle{plain} 
\newtheorem{theorem}{Theorem}[section]
\newtheorem{lemma}[theorem]{Lemma}
\newtheorem{proposition}[theorem]{Proposition}
\newtheorem{definition}[theorem]{Definition}

\theoremstyle{definition} 

\usepackage{authblk}

\makeatletter
\newcommand\CorrespondingAuthor[1]{%
  \begingroup%
  \def\@makefnmark{}%
  \footnotetext{Corresponding author: #1}%
  \endgroup%
}
\makeatother

\makeatletter
\renewenvironment{abstract}{%
  \small%
  \providecommand\keywords{%
    \par\medskip\noindent\textit{Keywords:}\xspace}%
  \begin{center}%
    \bfseries \abstractname\vspace{-.5em}\vspace{\z@}%
  \end{center}%
  \quote%
}{\endquote}
\makeatother

\usepackage[above,below,section]{placeins}

\usepackage{siunitx}

\usepackage[all]{onlyamsmath}

\usepackage{lipsum}

\usepackage[utf8]{inputenc}
\usepackage{float,multicol,eso-pic,rotating,bbm}

\newtheorem{corollary}[theorem]{Corollary}

\newcommand{\R}{\mathbb{R}}

\newcommand{\N}{\mathbb{N}}
\newcommand{\Z}{\mathbb{Z}}

\newcommand{\Me}{\mathrm{Me}}

\newcommand{\Mee}{\widehat{\mathrm{Me}}}

\newcommand{\Fe}{\widehat{F}}

\newcommand{\dd}{\mathrm d}
\newcommand{\1}{\mathbf 1}

\newcommand{\bX}{\mathbf{X}}
\newcommand{\bY}{\mathbf{Y}}
\newcommand{\bx}{\mathbf{x}}
\newcommand{\by}{\mathbf{y}}

\newcommand{\bZ}{\mathbf{Z}}

\renewcommand{\P}{\mathrm{P}}

\DeclareMathOperator\E{E}

\DeclareMathOperator\Var{Var}
\DeclareMathOperator\Cov{Cov}

\let\oldtextbf\textbf
\renewcommand\textbf[1]{\oldtextbf{\boldmath #1}}

\begin{document}
\title{Median-based estimation of the intensity of a spatial point process}
\author[1]{Jean-François Coeurjolly}
\affil[1]{Univ. Grenoble Alpes, F-38000 Grenoble, France\\
\texttt{Jean-Francois.Coeurjolly@upmf-grenoble.fr}.}

\date{}

\maketitle

\begin{abstract}
  This paper is concerned with a robust estimator of the intensity of a stationary spatial point process. The estimator corresponds to the median of a jittered sample of the number of points, computed from a tessellation of the observation domain. We show that  this median-based estimator satisfies a Bahadur representation from which we deduce its consistency and asymptotic normality under mild assumptions on the spatial point process. Through a simulation study, we compare the new estimator, in particular, with the standard one counting the mean number of points per unit volume. The empirical study confirms the asymptotic properties established in the theoretical part and shows that the median-based estimator is more robust to outliers than  standard procedures.

\end{abstract}

\section{Introduction}

Spatial point patterns are datasets containing the random locations of some event of interest. These datasets arise in many scientific fields such as biology, epidemiology, seismology, hydrology. Spatial point processes are the stochastic models generating such data. We refer to  \citet{stoyan:kendall:mecke:95}, \cite{illian:et:al:08} or \citet{moeller:waagepetersen:04} for an overview on spatial point processes. These references cover practical as well as theoretical aspects. A point process $\bX$ in $\R^d$ is a locally finite random subset of $\R^d$ meaning that the restriction to any bounded Borel set is finite. The point process $\bX$ takes values in $\Omega$, which contains in  all locally finite subsets of $\R^d$. Thus, the distribution of $\bX$ is a probability measure  on an appropriate $\sigma$-algebra  of $\Omega$. The Poisson point process is the reference process to model random locations of points without  interaction. Many alternative models such as Cox point processes, determinantal point processes, Gibbs point processes allow us to introduce clustering effects or to produce regular patterns (see again e.g. \citet{moeller:waagepetersen:04} and \citet{lavancier:moller:rubak:14} for an overview). If the distribution of $\bX$ is invariant by translation, we say that $\bX$ is stationary. We are interested in this paper in first-order characteristics of $\bX$, which under the assumption of stationarity, reduce to a single real parameter denoted by $\lambda$. This intensity $\lambda$ measures the average number of points per unit volume.

Estimating $\lambda$ is  a well-known problem and has been the subject of a large literature. Based on a single realization of the point process $\bX$ in a bounded domain $W$ of $\R^d$, the natural way of estimating $\lambda$ is to compute the average number of points observed by unit volume, i.e. to evaluate $N(W) /|W|$ where $N(W)$ denotes the number of points of $\bX$ falling into the observation domain $W$ with volume $|W|$. We denote this estimator by $\widehat\lambda^{\mathrm{std}}$ for this estimator. If the point process is a homogeneous Poisson point process, $\widehat\lambda^{\mathrm{std}}$ is also the maximum likelihood estimator. Asymptotic properties of $\widehat\lambda^{\mathrm{std}}$ are  well established for a large class of models. In particular, as the observation window expands to $\R^d$, it can be shown  under mild assumptions on $\bX$ (mainly mixing conditions) that $\widehat\lambda^{\mathrm{std}}$ is consistent and satisfies a central limit theorem with asymptotic variance, which can be consistently estimated (see \cite{heinrich:prokevsova:10} and the references therein for more details). 
In some applications, it may be too time-consuming to count all points. In such situations, distance based methods, where mainly nearest distances between points are used, have been developed (see e.g. \citet{byth:82,diggle:03,magnussen:12}). Unlike the estimator $\widehat\lambda^{\mathrm{std}}$, those methods are quite sensitive to the model, which  explains why the only practically applicable  case  is the Poisson process \citep{illian:et:al:08}. 
Other moment-based methods include the adapted estimator proposed by~\citet{mrkvivcka:molchanov:05} or the recent Stein estimator (in the Poisson case) proposed by~\cite{clausel:coeurjolly:lelong:15}.

As outlined in particular in the book  \citet{illian:et:al:08}, an important step in the statistical analysis of point patterns is the search for unusual points or unusual point configurations, i.e. the search of outliers. Two kinds of outliers appear when dealing with point pattern: first points may appear at locations where they are not expected. This situation  could appear for instance when two species of plants or trees cannot be distinguished at the time of data collection. Second, it is possible that there are  missing points in the pattern, i.e. areas of the observation domain where, according to the general structure of the pattern, points are expected. \citet{illian:et:al:08} or \citet{baddeley:etal:05} have proposed several diagnostic tools to detect outliers and more generally to judge the quality of fit of a model.  To the best of our knowledge, the works by \citet{berndt:stoyan:97} and \citet{assuncao:guttorp:99} are the only works where robustness of estimation procedures are tackled. \citet{assuncao:guttorp:99} developed an M-estimator to estimate the intensity of an inhomogeneous Poisson point process. For an application in materials science, \citet{berndt:stoyan:97} proposed the following methodology to estimate the intensity parameter of a stationary point process: let $\mathcal C$ be a typical cell of the Voronoi tessellation built from a stationary point process. It is known, see e.g. \cite{stoyan:kendall:mecke:95,moller:94}, that $\E |\mathcal C|= 1/\lambda$, whereby an estimator of $\lambda$ can be deduced by evaluating the sample mean of cell areas produced by the Voronoi tessellation of the observed point pattern. \citet{berndt:stoyan:97} proposed to replace the sample mean by a more robust estimator like a sample median or a trimmed-mean. \citet[p.~252]{illian:et:al:08} have suggested a slightly different procedure. Let $G=\{g_1,\dots,g_{\#G}\}$ be a grid of $\#G$ dummy points and let $a(g,\bX)$ be the cell area of the closest point $g\in G$ in $\bX$. Starting from the fact that $(\#G)^{-1}\sum_{g\in G} a(g,\bX)^{-1}$ is an unbiased estimator of $\lambda$ (up to edge effects), then a robust estimator can be constructed by replacing the sample mean by a trimmed-mean for instance (see Section~\ref{sec:sim} for more details). The two latter procedures described, which were not supported by theoretical results, are  the closest to the present work and we discuss them in Section~\ref{sec:sim}. 

As far as we know, no model free robust techniques supported by theoretical results have been developed. In this paper, we aim at developing a simple median-based estimator of $\lambda$. It is not so straightforward to see what a median means for a spatial point process but we may remark that if $W$ is decomposed as a union of $K$ non-overlapping cells $C_k$, then $N(W) = \sum_k N( C_k)$, which yields that $\widehat \lambda^{\mathrm{std}}$ can be actually rewritten as the empirical mean of the normalized counts variables $N( C_k)/|C_k|$. We have the cornerstone to define a more robust estimator by simply replacing the sample mean by the sample median.

The classical definition of sample quantiles and their asymptotic properties for continuous distributions are nowadays well-known,  see e.g. \citet{david:nagaraja:03}. In particular, sample medians in the i.i.d. setting, computed from an absolutely continuous distribution, $f$, positive at the true median, $\Me$, are consistent and satisfy a central limit theorem with asymptotic variance $1/4f(\Me)^2$. Such a result obviously fails for discrete distributions. In this paper, we follow a strategy introduced by~\citet{stevens:50} and applied to count data by~\citet{machado:silva:05} which consists in artificially imposing smoothness in the problem through jittering: i.e. we add to each count variable $N( C_k)$ a random variable $U_k$ following a uniform distribution on $(0,1)$. Now, the random variable $N( C_k)+U_k$ admits  a density and asymptotic results can be expected. To get around the  problem of large sample behavior for discrete distributions, another approach could be to consider the median based on the mid-distribution, see~\cite{ma:genton:parzen:11}. The authors prove that such sample quantiles behave more favourably than the classical one and satisfy, in the i.i.d. setting, a central limit theorem even if the distribution is discrete. We leave to a future work the question of deriving asymptotic properties for the sample median based on the mid-distribution in the context of this paper.

The rest of the paper is organized as follows. Section~\ref{sec:background} gives a short background on spatial point processes. General notation as well as the definition of our estimator are presented in Section~\ref{sec:estimators}. We also examine in Section~\ref{sec:estimators} how far the true median of $N(C_k)+U_k$ is from the intensity $\lambda|C_k|$. Section~\ref{sec:asymp} contains our main asymptotic results. General assumptions are discussed and a particular focus on Cox point processes is investigated. The main difficulty here is to establish a Bahadur representation for the jittered sample median, which can be applied to a large class of  models.  Section~\ref{sec:sim} presents the results of a  simulation study where we compare our procedure with the standard estimator $\widehat \lambda^{\mathrm{std}}$ and the estimator proposed by \citet{berndt:stoyan:97}. The research contained in this paper leads to a number of  interesting open questions, which are mentioned in Section~\ref{sec:conclusion}. Proofs of the results and additional  comments are postponed to Appendices~\ref{sec:proofs} and~\ref{app:supplementary}

\section{Background on spatial point processes} \label{sec:background}


Let $\bX$ be a spatial point process defined on $\R^d$, which
we see as a random locally finite subset of $\R^d$. Let $W$ be a bounded Borel set of $\R^d$, then the number of points in $\bX\cap W$, denoted by $N(W)$, is finite, and a realization of $\bX\cap W$ is of the form
$\bx=\{x_1,\dots,x_m\}\subset W$ for some nonnegative finite integer $m$. If $m=0$, then $\bx=\emptyset$ is an empty point pattern in $W$. For $u\in \R^d$, we denote by $\|u\|$ its Euclidean norm. For further background and measure theory on spatial point processes, see e.g.\ \cite{daley:vere-jones:03} and
\cite{moeller:waagepetersen:04}. We assume that $\bX$ is a stationary point process with intensity  $\lambda$, which, by Campbell's theorem (see e.g.\ \cite{moeller:waagepetersen:04}), is characterized by the fact that for any real Borel function $h$ defined on $\R^d$ and absolutely
integrable (with respect to the Lebesgue measure on
$\R^d$)
\begin{equation}\label{e:campbell}
\E\sum_{u\in\bX}h(u)= \lambda\int h(u)\dd u.
\end{equation}
Furthermore, for any integer $l\ge 1$, $\bX$ is said to have an $l$th-order product density
$\rho_l$ if $\rho_l$ is a non-negative Borel function on $\R^{dl}$
  such that for all non-negative Borel functions $h$ defined on $\R^{dl}$,
\begin{equation}\label{e:campbell2}
\E
\mathop{\sum\nolimits\sp{\ne}}_{u_1,\dots,u_l\in\bX} 
h(u_1,\dots,u_l) =\int_{\R^d}\cdots\int_{\R^d}
h(u_1,\dots,u_l)\rho_l(u_1,\dots,u_l)
\,\mathrm du_1\cdots\,\mathrm du_l,
\end{equation}
 where
the sign $\not=$ over the summation means that $u_1,\dots,u_l$ are
pairwise distinct. Note that $\lambda=\rho_1$ and that for the homogeneous Poisson point process $\rho_l(u_1,\dots,u_l)=\lambda^l$. If $\rho^{(2)}$ exists, then by the stationarity of $\bX$, $\rho^{(2)}(u,v)$ depends only on $u-v$. In that case, we define the pair correlation function $g$ as a function from $\R^d$ to $\R^+$ by $g(u-v)= \rho^{(2)}(u,v)/\lambda^2$. 

In this paper, we sometimes pay attention to Cox point processes, which are defined as follows.
\begin{definition}\label{def:cox} Let $(\xi(s),s\in \R^d)$ be a non-negative locally integrable random field. Then, $\bX$ is a Cox point process if the distribution of $\bX$ given $\xi$ is an inhomogeneous Poisson point process with intensity function $\xi$. If $\xi$ is stationary, so is $\bX$ and $\lambda = \E(\xi(s))$ for any $s$.
\end{definition}

Among often used models of stationary Cox point processes, we can cite
\begin{itemize}
\item Log-Gaussian Cox processes (e.g. \cite{moeller:waagepetersen:04}): 
Let $Y$ be a stationary Gaussian process on $\R^d$
  with mean  $\mu$ and
  stationary covariance function 
$c(u)=\sigma^2r(u)$, $u\in\R^d$, where
  $\sigma^2>0$ is the variance and $r$ the
  correlation function. If 
$\bX$ conditional on $Y$ is a Poisson point process with intensity
function $\xi=\exp(Y)$, then $\bX$ is a (homogeneous)
log-Gaussian Cox process. One example of correlation function is the Matérn correlation function (which includes the exponential correlation function)  given by $r(u)=(\sqrt{2\nu} \|u\|/\phi)^\nu K_\nu(\sqrt{2\nu} \|u\|/\phi)/(2^{\nu-1}\Gamma(\nu))$ where $\Gamma$ is the gamma function, $K_\nu$ is the modified Bessel function of the second kind, and $\phi$ and $\nu$ are non-negative parameters.
In particular, the intensity of $\bX$ equals $\lambda=e^{\mu+\sigma^2/2}$.
\item Neyman-Scott processes
  (e.g. \cite{moeller:waagepetersen:04}): Let  
$\mathbf C$ be a stationary Poisson point process with intensity $\kappa>0$, 
and $f_\sigma$ a
density function on $\R^d$. If 
$\bX$ conditional on $\mathbf C$ is a Poisson point process with intensity  
\begin{equation}\label{e:inhomSNCP}
\alpha \sum_{c\in\mathbf C}f_\sigma(u-c)/\kappa,\quad u\in\R^2,
\end{equation}
for some $\alpha>0$, then $\bX$ is a (homogeneous) Neyman-Scott process.
When $f_\sigma$ corresponds to the density of a uniform distribution on $B(0,\sigma^2)$ (resp. a Gaussian random variable with mean 0 and variance $\sigma^2$), we refer to $\bX$ as the (homogeneous) Matérn Cluser (resp. Thomas) point process. In particular, the intensity of $\bX$ equals $\lambda=\alpha \kappa$.
\end{itemize}

\section{Median-based estimator of $\lambda$} \label{sec:estimators}


For any real-valued random variable $Y$, we denote by $F_Y(\cdot)$ its cdf, by $F_Y^{-1}(p)$ its quantile of order $p\in (0,1)$, by $\Me_Y=F_Y^{-1}(1/2)$ its theoretical median. Based on a sample $\mathbf{Y}=(Y_1,\dots,Y_n)$ of $n$ identically distributed random variables we denote by $\Fe(\cdot;\mathbf Y)$ the empirical cdf, by $\Fe^{-1}(p;\mathbf Y)$ the sample quantile of order $p$ given by
\begin{equation} \label{def:quantile}
  \Fe^{-1}(p;\mathbf Y) = \inf \{ x\in \R: p \leq \Fe(x;\bY)\}.
\end{equation}
The sample median is simply denoted by $\Mee(\mathbf Y)=\Fe^{-1}(1/2;\mathbf Y)$.

We will study the large-sample behavior of estimators of the intensity $\lambda$. Specifically, we consider a region $W_n$ assumed to increase to $\R^d$ as $n\to \infty$. We assume that the domain of observation $W_n$ can be decomposed as $W_n= \cup_{k\in \mathcal K_n} C_{n,k}$ where the cells $C_{n,k}$ are non-overlapping and equally sized with  volume $c_n=|C_{n,k}|$ and where $\mathcal K_n$ is a subset of $\Z^d$ with cardinality $k_n=|\mathcal K_n|$. More details on $W_n, c_n$ and $k_n$ will be provided in the appropriate Section~\ref{sec:asymp} when we present asymptotic results. Finally, for any random variable $Y$ or any random vector $\bY$, we denote by $\check Y=Y/c_n$ and $\check \bY = \bY/c_n$. 

The classical estimator of the intensity $\lambda$ is given by $\widehat \lambda^{\mathrm{std}} = N(W_n)/|W_n|$. In order to define a more robust estimator, we can note that
\begin{equation}\label{eq:mean}
  \widehat\lambda^{\mathrm{std}} = \frac1{k_n} \; \sum_{k\in \mathcal K_n} \frac{N(C_{n,k})}{c_n} 
\end{equation}
since $|W_n|=k_n c_n$, i.e. $\widehat \lambda^{\mathrm{std}}$ is nothing else that the sample mean of intensity estimators computed in cells $C_{n,k}$. The strategy adopted in this paper is to replace the sample mean by the sample median, which is known to be more robust to outliers. As underlined in the introduction, quantile estimators based on count data or more generally on discrete data can cause some troubles in the asymptotic theory. The problems come from the fact that,  in the continuous case, the asymptotic variance of the sample median involves the probability distribution function at the true median. 

To overcome the problem of discontinuity of the counts variables $N(C_{n,k})$, we follow a well-known technique (e.g. \citet{machado:silva:05}) which introduces smoothness in the problem. Let $(U_k, k\in \mathcal K_n)$ be a collection of independent and identically distributed  random variables, distributed as $U\sim \mathcal U([0,1])$. Then, for any $k\in \mathcal K_n$, we define 
\begin{equation}\label{eq:Znk}
  Z_{n,k}  = N(C_{n,k}) + U_k \quad \mbox{ and } \quad 
  \bZ=(Z_{n,k}, \; k\in \mathcal K_n).
\end{equation}
Since $\bX$ is stationary, the variables $Z_{n,k}$ are identically distributed and we let $Z\sim Z_{n,k}$. 
The jittering effect shows up right away: the cdf of $Z$ is given for any $t\geq 0$ by
\[
  F_{Z}(t) = P( N(C_{n,0})\leq  \lfloor t \rfloor -1) + P(N(C_{n,0})=\lfloor t\rfloor ) \,(t-\lfloor t \rfloor),
\]
which is a continuously differentiable function whereby we deduce that $Z$ admits a density $f_Z$ at $t$ given by
\begin{equation}\label{eq:fZ}
  f_Z(t) = P(N(C_{n,0})=\lfloor t\rfloor ).
\end{equation}
We  define the jittered estimator of the intensity $\lambda$ by
\begin{equation} \label{eq:lambdaJ}
  \widehat \lambda^J = \Mee ( \check \bZ) = \frac{\Mee (\bZ)}{c_n}
\end{equation}
where the sample median is defined by~\eqref{def:quantile}.
Since it is expected that $\Mee(\check \bZ)$ is close to $\Me_{\check Z}= \Me_Z/c_n$, we need to understand how far $\Me_{\check Z}$ is from $\lambda$. Using the definition of the median we can prove the following result.

\begin{proposition} \label{prop:MeZ}
  Assume that the pair correlation function of the stationary point process $\bX$ exists for $u,v\in \R^d$ and satisfies  $\int_{\R^d} |g(w)-1|\dd w<\infty$, then for any $\varepsilon>0$ we have for $n$ sufficiently large
  \begin{equation}\label{eq:MeZcheck}
    | \Me_{\check Z} - \lambda | \leq \frac 1{c_n}\left(\frac12+\sqrt{\frac1{12}}\right) + (1+\varepsilon)\sqrt{ \frac{\sigma}{c_n}} = \mathcal O (c_n^{-1/2})
  \end{equation}
  where $\sigma^2= \lambda+\lambda^2\int_{\R^d}(g(w)-1)\dd w$.
  \end{proposition}
 The assumption $\int|g(w)-1|\dd w<\infty$ is quite standard when we deal with asymptotics for spatial point processes, see e.g. \cite{guan:loh:07} or \cite{heinrich:prokevsova:10}. It ensures that for any sequence of regular domains $\Delta_n$, $|\Delta_n|^{-1} \Var(N(\Delta_n)) \to \sigma^2$ as $n\to \infty$. We refer the reader to these papers and to Section~\ref{sec:assumptions} for a discussion on this assumption.
\begin{proof}
  By the previous remark, ${c_n}^{-1}\Var( N(C_{n,k}))  \to \sigma^2$ as $n\to \infty$. Since for any continuous random variable $Y$ with finite first two moments, $|\Me_Y-\E(Y )|\leq \sqrt{\Var(Y)}$ and since $\E(Z)=\lambda c_n+1/2$, then for any $\varepsilon>0$ we have for $n$ sufficiently large
  \[
    |\Me_Z-\lambda c_n| \leq \frac12  + \sqrt{\frac{1}{12} + (1+\varepsilon)^2 \sigma^2 c_n}.
  \]
Dividing both sides of the inequality by $c_n$ leads to~\eqref{eq:MeZcheck}.
  \end{proof}
Let $\Pi\sim \mathcal P(\nu)$ be a Poisson random variable with mean $\nu>0$. 
Several results are known for the theoretical median of $\Pi$, see e.g. \cite{adell:jodra:05}. 
For instance, when $\nu$ is an integer $\Me_{\Pi} = \nu$ and for non integer $\nu$,  $-\log 2 \leq \Me_{\Pi} \leq 1/3$ (see Figure \ref{fig:medianPoisson}). Based on this, we can obtain a sharper inequality than~\eqref{eq:MeZcheck} for Poisson and Cox point processes.
\begin{proposition}\label{prop:MeZcheckCox}
Let $\bX$ be a stationary Cox point process with latent random field $\xi$, then
\begin{equation}
\label{eq:MeZcheckCox}
\lambda c_n -\log 2 \leq \Me_{N(C_{n,0})} \leq \lambda c_n +\frac13 
\quad \mbox{ and } \quad
|\Me_{ Z} - \lambda c_n| \leq \frac4{3}.
\end{equation}
\end{proposition}
A reformulation of \eqref{eq:MeZcheckCox} is of course $\Me_{\check Z}-\lambda =\mathcal O(c_n^{-1})$.

\begin{proof} Given $\xi$, for any $k\in \mathcal K_n$, $N(C_{n,k})$ follows a Poisson distribution with mean $\int_{C_{n,k}}\xi(s)\dd s$. Denote by $\Me_{N(C_{n,k})\mid \xi}$  the median of $N(C_{n,k})$ given $\xi$ defined by
\[
  \Me_{N(C_{n,k})\mid \xi} = \inf \left\{ z\in \R: 
F_{N(C_{n,k})\mid \xi} (z) \geq 1/2
  \right\}
\]
where $ F_{N(C_{n,k})\mid \xi}$ is the  cumulative distribution function of $N(C_{n,k})$ given $\xi$. From the property of the median of a Poisson distribution, we have for any $k\in \mathcal K_n$ 
\[
\int_{C_{n,k}} \xi(s) \dd s -\log 2 \leq \Me_{N_{n,k}\mid \xi} \leq \int_{C_{n,k}}\xi(s)\dd s +\frac13.
\] 
Since $\E \int_{C_{n,k}} \xi(s)\dd s=\lambda c_n$, the first result is deduced by 
taking the expectation of each term of the previous inequality. 
Since $N(C_{n,0}) \leq Z\leq N(C_{n,0}) +1$, $\Me_{N(C_{n,0})} \leq \Me_Z \leq \Me_{N(C_{n,0})}+1$ which leads to the second result.
\end{proof}

\section{Asymptotic results} \label{sec:asymp}

We state in this section our main results and the general assumptions required to obtain them. Proofs of the different results presented here, as well as auxiliary results, are presented in Appendix~\ref{sec:proofs}.

\subsection{General assumptions and discussion} \label{sec:assumptions}

We recall the classical definition of mixing coefficients
(see e.g.\ \cite{politis:98}): for $j,k \in \N\cup \{\infty\}$ and
$m\geq 1$, define 
\begin{align}
  \alpha_{j,k}(m)=\sup \{&  
|P(A\cap B) - P(A)P(B)|:\, A\in \mathcal{F}(\Lambda_1),\, B\in \mathcal{F}(\Lambda_2), \nonumber\\
& \Lambda_1\in \mathcal B(\R^d),\,\Lambda_2 \in \mathcal B(\R^d),\, |\Lambda_1|\leq j,\, |\Lambda_2|\leq k,\, d(\Lambda_1,\Lambda_2)\geq m
  \} \label{def:mixing}
\end{align}
where $\mathcal{F}(\Lambda_i)$ is the $\sigma$-algebra generated by
$\bX\cap \Lambda_i$, $i=1,2$, $d(\Lambda_1,\Lambda_2)$ is the
minimal distance between the sets $\Lambda_1$ and $\Lambda_2$, and
$\mathcal B(\R^d)$ denotes the class of Borel sets in~$\R^d$.

We require the following assumptions to prove our asymptotic results.\\

\noindent (i) For any $n\geq 1$, we assume that $W_n=\cup_{k\in \mathcal K_n}C_{n,k}$ where $\mathcal K_n$ is a subset of $\Z^d$ with cardinality
 $k_n=|\mathcal K_n|$ and where the cells $C_{n,k}$ are equally sized and non-overlapping cubes with volume $c_n$, defined by 
\[
  C_{n,k}=\left\{u=(u_1,\dots,u_d)^\top\in \R^d: c_n^{1/d}(k_l-1/2)\leq u_l\leq c_n^{1/d}(k_l+1/2), \;l=1\dots,d \right\}.
\]
We assume that $0\in \mathcal K_n$ and that there exists $0<\eta^\prime<\eta$ such that as $n\to \infty$ 
\[
  k_n\to \infty, \quad c_n\to \infty,  \quad \frac{k_n}{c_n^{\eta^\prime/2 \wedge (1-2\ell)}}\to 0
\]
where $\ell$ is given by Assumption (ii) and  $\eta$  by Assumption~(iv).
 
\noindent (ii) ${ }$\\
\hspace*{.5cm} \begin{minipage}{14cm}\begin{itemize}
  \item[(ii-1)] $\Me_Z-\lambda c_n =\mathcal O({c_n}^{\ell})$ with $0\leq\ell < 1/2$.
  \item[(ii-2)] $\forall t_n= \lambda c_n + \mathcal O(\sqrt{c_n/k_n})$,
  $
     {\P(N(C_{n,0})=\lfloor t_n \rfloor)}/{ \P(N(C_{n,0})=\lfloor \lambda c_n \rfloor)}  \to 1.
  $
  \item[(ii-3)] There exist $\underline\kappa,\overline\kappa>0$ such that for $n$ large enough, 
  $
    \underline\kappa \leq \sqrt{c_n} f_Z(\Me_Z) \leq \overline\kappa.
  $
\end{itemize} \end{minipage}

\noindent (iii) $\bX$ has a pair correlation function $g$ satisfying 
$
  \int_{\R^d} |g(w)-1| \dd w <\infty.
$\\
\noindent (iv) There exists $\eta >0$ such that
\[
  \alpha(m)= \sup_{p\geq 1} \frac{\alpha_{p,p}(m)}{p} = \mathcal O(m^{-d(1+\eta)}) \quad \mbox{ and } \quad
  \alpha_{2,\infty}(m) = \mathcal O(m^{-d(1+\eta)})
\]
where $\alpha_{j,k}(m)$ for $j,k\in \N\cup \{\infty\}$ is defined by~\eqref{def:mixing}.

Now, we discuss the different assumptions. The last statement of Assumption (i) is required to control the dependency between the variables $Z_{n,k}$ through the control of the mixing coefficients and to ensure that asymptotically  $|W_n|^{1/2}(\widehat\lambda^J-\lambda)$ behaves as $|W_n|^{1/2}(\widehat\lambda^J - \Me_{\check Z})$. We note that if $\bX$ is a stationary Cox point process, Proposition~\ref{prop:MeZcheckCox} yields that (ii-1) is satisfied for $\ell=0$. So if $\eta>2$, Assumption (i) can be rewritten as $c_n\to \infty$, $k_n\to \infty$ and $k_n /c_n\to 0$ as $n\to \infty$.  

Regarding Assumption~(ii), Proposition~\ref{prop:iiCox} stated below shows it can be simplified for a large class of Cox point processes.
We underline that Assumptions (i), (ii-1)-(ii-2)  imply the existence of $\overline\kappa<\infty$ such that $\sqrt{c_n} f_Z(\Me_Z)\leq \overline\kappa$, so (ii-3) could actually be simplified.

Assumption~(iii) is  classical when dealing with asymptotics of intensity  estimates, see e.g. \citet{heinrich:prokevsova:10}. For isotropic pair correlation functions, i.e. $g(w)=\overline g(\|w\|)$ for $\overline g:\R^+ \to\R$, Assumption (iii) is fulfilled when $\overline g(r)=0$ for $r\geq R$ or when $\overline g(r)=\mathcal O(r^{-d-\gamma})$ for some $\gamma>d$. This includes the Matérn cluster and Thomas processes and the log-Gaussian Cox process with Matérn-Whittle covariance functions. 

Assumption~(iv) is also quite standard and has been discussed a lot in the literature: \citet{guan:loh:07,guan:sherman:calvin:07,prokevsova:jensen:13} discussed the first part of~(iv) while the second one was commented in \cite{waagepetersen:guan:09,coeurjolly:moller:14}. Both of them are satisfied for the Matérn Cluster and Thomas processes and for log Gaussian Cox processes with correlation function decaying fast enough to zero. 

We point out that it is not so common to use both the mixing coefficients $\alpha(m)$ and $\alpha_{2,\infty}(m)$. As detailed in the proof of Theorem~\ref{thm:cltF}, the first one is used to control the dependence between the random variables $Z_{n,k}$ for $k\in \mathcal K_n$ and derive a central limit theorem using the blocking technique developed by \cite{ibragimov:linnik:71} which is pertinent and appropriate here since the cells $C_{n,k}$ are increasing. The second mixing coefficient is necessary to apply a multivariate central limit theorem inside the cell $C_{n,0}$. We prove in particular that $\P(N( {C_{n,0}})\leq \lambda c_n,N(C_{n,0}^-)\leq \lambda c_n )\to 1/2$ as $n\to \infty$ where $C_{n,0}^-$ is a "small" erosion of $C_{n,0}$ (see the proof of Step 1 of Theorem~\ref{thm:cltF} for more details).

The next result shows the simplifications we can obtain for Cox point processes.
\begin{proposition} \label{prop:iiCox}
  Let $\bX$ be a stationary Cox point process with latent random field $(\xi(s),s\in \R^d)$ satisfying the Assumptions (iii)-(iv). Assume there exists $\delta>2/\eta$, where $\eta$ is given by Assumption~(iv), such that $\E(|\xi(0)|^{2+\delta})<\infty$. Let $t_n=\lambda c_n +\mathcal O(\sqrt{c_n/k_n})$ and $T_n=\lfloor t_n \rfloor^{-1} \int_{C_{n,0}}\xi(s)\dd s$. We also assume that the sequence of  random variables $(B_n)_n$ defined by 
$
\log(B_n)={\lfloor t_n \rfloor \left( 
\log(T_n) -(T_n-1) +{(T_n-1)^2}/2 
\right) }
$ 
is uniformly integrable. Then, Assumption~(ii) holds (with $\ell=0$) and as $n\to \infty$
  \begin{equation}\label{eq:ii23Cox}
      \sqrt{c_n} \P( N(C_{n,0}) =\lfloor \lambda c_n \rfloor) \to \left( 2\pi \sigma^2\right)^{-1/2}
  \end{equation}  
where $\sigma^2 = \lambda + \lambda^2 \int_{\R^d} (g(w)-1)\dd w$.
\end{proposition}

\subsection{Results}

In this section, we present  asymptotic results  for $\Fe(\cdot;\bZ)$, the empirical cumulative distribution function based on $\bZ$  and for the median-based estimator $\widehat\lambda^J$.

\begin{theorem}\label{thm:cltF}
Under the Assumptions~(i)-(iv), we have the  following two statements.\\
(a) Let $(a_n)_{n\geq 1}$ be a sequence of real numbers satisfying $\lambda c_n =a_n +o(\sqrt{c_n})$, then as $n\to \infty$
\[
  \sqrt{k_n} \left( \Fe (\lambda c_n+a_n;\bZ) - F_Z(\lambda c_n +a_n)\right) \to \mathcal N(0,1/4)
\]
in distribution.\\
(b) As $n\to \infty$
\[
    \sqrt{k_n} \left( \Fe (\Me_{Z};\bZ) -1/2\right) \to \mathcal N(0,1/4)
\]  
in distribution.
\end{theorem}
If $\bZ$ corresponds to a sample of $n$ i.i.d. random variables, $\sqrt n (\Fe(p;\bZ)-p)$ tends to a Gaussian random variable with mean zero and variance $p(1-p)$. Hence, we recover the same result as in our dependency setting.

The next result establishes a Bahadur representation for the sample median, leading to its asymptotic normality. The notation $X_n = o_\P(v_n^{-1})$ for a sequence of random variables $X_n$ and a sequence of positive real numbers $v_n$ means that $v_n X_n$ tends to 0 in probability as $n\to \infty$.

\begin{theorem}  \label{thm:bahadur}
Under the assumptions (i)-(iv), we have the  following two statements.\\
(a) As $n\to \infty$
\begin{equation}
  \label{eq:bahadur}
  \Mee(\bZ) - \Me_Z \; = \; \frac{1/2-\Fe(\Me_Z ; \bZ)}{f_Z(\Me_Z)} \; + \;  o_{\P} \left( \sqrt{\frac{c_n}{k_n}}\right).
\end{equation}
(b) Let $s_n = \sqrt{c_n} \P(N(C_{n,0})=\lfloor \lambda c_n \rfloor)$, then as $n\to \infty$
\begin{equation}
  \label{eq:convEst}
  |W_n|^{1/2} s_n \left( \widehat \lambda^J - \lambda \right) \to \mathcal N (0,1/4)
\end{equation}
in distribution.
\end{theorem}

\noindent We deduce the following Corollary given without proof for Cox point processes.

\begin{corollary} \label{cor:lambdaEstCox}
Under the Assumption (i) and the Assumptions of Proposition~\ref{prop:iiCox}, we have as $n\to \infty$
\[
  |W_n|^{1/2 } \left( \widehat \lambda^J - \lambda \right) \to \mathcal N\left(
  0,  \pi \sigma^2/2 \right)
\]
in distribution, where 
$\sigma^2 = \lambda + \lambda^2 \int_{\R^d} (g(w)-1)\dd w$.  
\end{corollary}
As detailed after Proposition \ref{prop:MeZ}, $\sigma^2$ corresponds to the asymptotic variance of $|W_n|^{-1} N({W_n})$. Actually, if we denote by $\widehat \lambda^{\mathrm{std}}$ the standard estimator of $\lambda$ given by $\widehat \lambda^{\mathrm{std}} = |W_n|^{-1} N({W_n})$ then with quite similar assumptions, it has been proved, see e.g. \citet{heinrich:prokevsova:10}, that $|W_n|^{1/2}(\widehat \lambda^{\mathrm{std}} - \lambda) \to \mathcal N(0,\sigma^2)$. It is worth noting that the two estimators $\widehat\lambda^{\mathrm{std}}$ and $\widehat\lambda^J$ only differ from their asymptotic variance and that the ratio of the asymptotic variances is equal to $\pi/2$. When we estimate the location of a Gaussian sample using the sample mean or the sample median, it is remarkable that the ratio of the asymptotic variances is also $\pi/2$.

Finally, let us add that on the basis of Corollary~\ref{cor:lambdaEstCox}, an asymptotic confidence interval of $\lambda$ can be constructed using a consistent estimator of $\sigma^2$. By the previous remark, we can use the kernel-based estimator proposed by \citet{heinrich:prokevsova:10} (or any other estimator presented in the mentioned paper), which precisely estimates the asymptotic variance of $\widehat\lambda^{\mathrm{std}}$, i.e. $\sigma^2$.

\section{Simulation study} \label{sec:sim}

We present in this section a simulation study where, in particular, we intend to compare  the median-based estimator defined by~\eqref{eq:lambdaJ} with the standard moment-based estimator $\widehat \lambda^{\mathrm{std}} = N(W)/|W|$. In the end of this section, we also investigate the robust estimator proposed by \cite{berndt:stoyan:97}. 

We focus on the planar case $d=2$. Three models of spatial point processes are considered:
\begin{itemize}
  \item Poisson point processes (referred to as {\sc poisson}) with intensity $\lambda$.
  \item Log-Gaussian Cox Processes (referred to as {\sc lgcp}) point processes with exponential covariance function. We fixed the variance to $0.5$ and $\phi$ to $0.02$. The parameter $\mu$ is fixed by the relation $\mu= \log \lambda -\sigma^2/2$ (see Section~\ref{sec:background} for details).
  \item Poisson hard-core (referred to as {\sc phc}) with parameter $\beta$ and hard core $R$. The resulting process consists in a conditional Poisson point process with intensity $\beta$ where we have conditioned on the hard core condition which consists in prohibiting points from being closer than distance $R$ apart. We fixed $\beta$ to 200 and $R$ to $0.05$.
\end{itemize}
The {\sc poisson} model is used as a benchmark. The {\sc lgcp} model enters into the class of Cox point processes for which we focused a lot in our asymptotic results. We have also considered the Thomas model in a separate simulation study and have obtained quite similar results to the {\sc lgcp} case. The {\sc phc} model is used as an example of repulsive point process model. Note that for fixed $\beta, R$ the intensity parameter $\lambda$ of the process is not explicit. In our simulation study, we fixed $\lambda=100$ for the {\sc poisson} and {\sc lgcp} models. With the parameters $\beta=200$ and $R=0.05$, we estimated  the intensity of a {\sc phc} to $\lambda \simeq 86$ using 10000 Monte-Carlo replications. 
The simulations have been performed using the \texttt{R} package \texttt{spatstat} \citep{baddeley:turner:05}. 

To illustrate the performances of~\eqref{eq:lambdaJ} we generated the point processes on the domain of observation $W_n=[-n,n]^2$ for different values of $n$ and considered the three following settings: let $\by$ be a realization from one of the three models described above, generated on $W_n$ and with $m$ points. The observed point pattern is denoted by $\bx$ and is obtained as follows.
\begin{itemize}
  \item[(A)] Pure case: no modification is considered, $\bx=\by$.
  \item[(B)] A few points are added: in a sub-square $\Delta_n$ with side-length $n/5$ included in $W_n$ and randomly chosen, we  generated a point process $\by^{\mathrm{add}}$ of $n^{\mathrm{add}}= \rho \, m$ uniform points in $\Delta_n$. We chose $\rho=0.05$ or $0.1$. Then, we defined $\bx=\by \cup \by^{\mathrm{add}}$.
  \item[(C)] A few points are deleted: let $\Delta_n = \cup_{q=1}^4 \Delta_n^q$ where the $\Delta_n^q$'s are the four equally sized squares included in $W_n$, located in each corner  of $W_n$. The volume of $\Delta_n$ is chosen such that $\E(N({\Delta_n})) = \rho \, \E(N(W_n)) = \rho \lambda |W_n|$ and we chose either $\rho=0.05$ or $0.1$. Then, we define $\bx= \by\setminus (\by\cap{\Delta_n})$, .i.e. $\bx$ is the initial configuration thinned by 5\% or 10\% of its points.
\end{itemize}
We conducted a Monte-Carlo simulation and generated $1000$ replications of the models {\sc poisson, lgcp, phc} and for the three different settings (A)-(C). The observation windows for which we report the empirical results hereafter are $n=1,2$. Regarding the setting (C), we placed the squares in which points are thinned at the corners of $W_n$. By stationarity, the empirical results are the same if we  decide to choose them randomly. For each replication, we evaluated $\widehat\lambda^{\mathrm{std}}$ and $\widehat\lambda^J$ for different number of non-overlapping and equally sized cells $k_n$. More precisely, we chose $k_n=9,16,25,36,49$.  

Empirical means and standard deviations related to the case (A) are reported in Table~\ref{tab:nothing}. We can check that the standard estimator is of course unbiased and that the standard deviation decreases by a factor close to 2, which is equal to $\sqrt{|W_2|/|W_1|}$. The median-based estimator is not theoretically unbiased but the bias is clearly not important and tends to decrease when the observation window grows up. Similarly, the rate of convergence of the empirical standard deviation is not too far from the expected value 2. We also computed separately $\widehat \Var(\widehat \lambda^{\mathrm{std}}) / \widehat \Var(\widehat \lambda^J)$ for each value of $k_n$ and $n$ and found interesting that these ratios are not too far from $\pi/2$. Finally, we underline that the choice of the number of cells $k_n$ has  little influence on the performances. When $n=1$, a too large value of $k_n$ seems to increase the bias, especially for the {\sc lgcp} and {\sc phc} model. The differences are however reduced when $n=2$. We also note that the empirical variance is almost the same whatever the value of $k_n$.

Tables~\ref{tab:add} and~\ref{tab:delete} are respectively related to the settings (B) and (C) (described above). Both these contaminations (B) or (C) can affect significantly the bias of the estimator. In both tables, we report the bias of the different estimators and the gain (in percent) in terms of mean squared error of the median based-estimator with respect to the standard one, i.e. for each model and each value of $\rho, n, k_n$, we computed 
\begin{equation}
    \label{eq:gain}
    \widehat{\mathrm{Gain}} =  \left( \frac{\widehat{\mathrm{MSE}}(\widehat\lambda^{\mathrm{std}}) - 
  \widehat{\mathrm{MSE}}(\widehat\lambda^{J})}{\widehat{\mathrm{MSE}}(\widehat\lambda^{\mathrm{std}})} \right) \times 100\%
  \end{equation}  
  where $\widehat{\mathrm{MSE}}$ is the empirical mean squared error based on the $1000$ replications. Thus a positive (resp. negative) empirical gain means that the median-based estimator is more efficient (resp. less efficient) than the standard procedure.

The standard estimator,  based only on the global number of points, is of course not robust to perturbations. It is clearly seen that $\widehat\lambda^{\mathrm{std}}$ has a positive bias when we add points (setting (B)) and a negative one when we delete points (setting (C)). This bias is obviously all the more important as $\rho$ (the ratio of points added or deleted) increases. Unlike this, the median-based estimator shows its advantages. When points are added (setting (B)), the estimator $\widehat \lambda^J$ remains much more stable and is more efficient in terms of MSE except when $\rho=0.05$ and $n=1$ for the three models where the empirical results do not lead to clear conclusions. When $n$ and/or $\rho$ increases the gains in percent are quite important and it is worth noticing that results do not fluctuate that much with the choice of the number of cells $k_n$. We also mention that, in a separate simulation study not shown here,  we  tried to add a clustered point process or repulsive point process, instead of adding  uniform points. The empirical results remained almost unchanged. 

Comments regarding Table~\ref{tab:delete} (setting (C)) are very similar. The results vary a lot when $n=1$ and $\rho=0.05$ but as soon as one of this parameter increases, the bias of the median-based estimators are clearly reduced, which implies they outperform significantly the standard estimator.

\begin{table}[H]
\centering
{\small\begin{tabular}{rrrrrrr}
  \hline
&  \multicolumn{6}{c}{Emprical mean (Standard Deviation)} \\
& \multicolumn{1}{c}{$\widehat\lambda^{\mathrm{std}}$} & \multicolumn{5}{c}{$\widehat \lambda^J$}\\
 &  & $k_n=9$ & $16$  & $25$ & $36$  & $49$ \\
  \hline
\multicolumn{3}{l}{ {\sc poisson} } &&\\
 $n=1$ &  99.6 (4.9) & 100.5 (5.9) & 101 (5.7) & 101.8 (6) & 102.8 (6.1) & 104 (6) \\
  $n=2$ & 99.9 (2.5) & 100.2 (3) & 100.3 (3) & 100.5 (3.1) & 100.5 (3.1) & 101 (3) \\
&&&&&&\\
\multicolumn{3}{l}{ {\sc lgcp}  }&&&\\
$n=1$&100.3 (5.5) & 101.2 (6.6) & 101.8 (6.4) & 102.3 (6.6) & 103 (6.8) & 104.2 (6.8) \\
$n=2$ & 100.1 (2.7) & 100.5 (3.2) & 100.4 (3.2) & 100.6 (3.3) & 100.7 (3.3) & 101 (3.3) \\
&&&&&&\\
\multicolumn{3}{l}{ {\sc phc} } &&&\\
  $n=1$ &86 (3) & 87.3 (4) & 87.7 (3.9) & 88.9 (4.1) & 90.1 (4.1) & 91.7 (4.2) \\
  $n=2$ &86 (1.6) & 86.3 (2) & 86.4 (2) & 86.7 (2) & 87 (2) & 87.4 (2.1) \\
   \hline
\end{tabular}
\caption{\label{tab:nothing} Empirical means and standard deviations between brackets of estimates of the intensity $\lambda$ for different models of spatial point processes ({\sc poisson, lgcp, phc}). The empirical results are based on 1000 replications simulated on $[-n,n]^2$ for $n=1,2$. The second and third columns correspond to the standard estimator $\widehat\lambda^{\mathrm{std}}=N({W_n})/|W_n|$, while the following ones correspond to the median-based estimator~\eqref{eq:lambdaJ} for different number of cells $k_n$. The intensity $\lambda$ equals 100 for the models {\sc poisson,lgcp} and (approximately) 86 for the model {\sc phc}.}
}
\end{table}

\begin{table}[H]
\centering
\begin{tabular}{rrrrrrr}
  \hline
 & \multicolumn{5}{c}{Bias (Gain of MSE \%)}\\
 & std & $k_n=9$ & $16$  & $25$ & $36$  & $49$ \\
  \hline
\multicolumn{1}{l}{$\rho=0.05$}  &&&&&&\\\
{\sc poisson}, $n=1$  &5.4 (0) & 4.2 (-14) & 4.5 (-9) & 4.8 (-19) & 6.3 (-53) & 7.1 (-75) \\
  $n=2$ & 5 (0) & 1.9 (50) & 1.8 (51) & 2.2 (47) & 2.3 (47) & 2.7 (43) \\ 
  {\sc lgcp}, $n=1$ & 5 (0) & 3.2 (-9) & 3.6 (-6) & 4.6 (-26) & 5.3 (-35) & 6.4 (-64) \\
  $n=2$ & 5.2 (0) & 2.2 (44) & 2.2 (49) & 2.2 (46) & 2.5 (47) & 2.8 (39) \\ 
   {\sc phc}, $n=1$ & 5 (0) & 3.3 (14) & 3.8 (5) & 4.9 (-19) & 6 (-53) & 7.6 (-110) \\
  $n=2$ &5 (0) & 1.4 (70) & 1.6 (72) & 1.8 (70) & 2.1 (66) & 2.5 (57) \\
 \hline
  \multicolumn{1}{l}{$\rho=0.1$}  &&&&&&\\\
   {\sc poisson}, $n=1$ &10.1 (0) & 4.7 (44) & 5 (45) & 5.7 (35) & 6.9 (29) & 7.8 (17) \\
  $n=2$ & 10.1 (0) & 2.6 (79) & 2.2 (84) & 2.5 (83) & 2.7 (83) & 2.8 (83) \\
  {\sc lgcp}, $n=1$ &9.8 (0) & 4.4 (38) & 5 (41) & 5.7 (30) & 6.3 (28) & 7.2 (21) \\
  $n=2$ &9.8 (0) & 2.2 (81) & 2.3 (79) & 2.2 (82) & 2.2 (81) & 2.5 (81) \\
   {\sc phc}, $n=1$ &10 (0) & 3.9 (64) & 4.1 (66) & 5.3 (57) & 6.5 (44) & 8 (24) \\
  $n=2$ &  10 (0) & 1.8 (89) & 1.7 (92) & 1.9 (91) & 2.1 (91) & 2.5 (89) \\
   \hline
\end{tabular}

\caption{\label{tab:add} Bias and empirical gains in percent between brackets, see~\eqref{eq:gain}, for the standard and  median based estimators for different values of $k_n$. The empirical results are based on 1000 replications generated on $[-n,n]^2$ for $n=1,2$ for the models {\sc poisson, lgcp, phc} where $5\%$ or $10\%$ of points are added to each configuration. This corresponds to the case (B) described in details above.}
\end{table}

\begin{table}[H]
\centering
\begin{tabular}{rrrrrrr}
  \hline
 & \multicolumn{5}{c}{Bias (Gain of MSE \%)}\\
 & std & $k_n=9$ & $16$  & $25$ & $36$  & $49$ \\
  \hline
\multicolumn{1}{l}{$\rho=0.05$}  &&&&&&\\\
{\sc poisson}, $n=1$  &
   -5 (0) & -3.9 (-13) & -2.8 (2) & -2 (8) & -0.7 (13) & 0.6 (11) \\
  $n=2$ & 
  -4.9 (0) & -4.2 (6) & -3 (34) & -2 (46) & -1 (59) & -0.4 (62) \\
  {\sc lgcp}, $n=1$ & 
 -5 (0) & -3.9 (0) & -3.1 (-2) & -2.1 (2) & -1.1 (17) & 0.1 (18) \\
  $n=2$ & 
 -5 (0) & -4.3 (3) & -3.2 (27) & -2 (46) & -1.2 (60) & -0.6 (62) \\
     {\sc phc}, $n=1$ & 
-4.2 (0) & -2.7 (11) & -1.6 (26) & 0 (25) & 1.6 (27) & 3.4 (-5) \\
  $n=2$ &
-4.3 (0) & -3.3 (22) & -1.8 (59) & -0.8 (72) & 0 (78) & 0.7 (75) \\
 \hline
  \multicolumn{1}{l}{$\rho=0.1$}  &&&&&&\\\
   {\sc poisson}, $n=1$ &
-10 (0) & -8.6 (1) & -6 (32) & -3.3 (53) & -1.1 (63) & -1.4 (65) \\
  $n=2$ & 
-10 (0) & -7.2 (34) & -3.1 (81) & -1.9 (88) & -1 (88) & -2.4 (86) \\
  {\sc lgcp}, $n=1$ &  
 -10.4 (0) & -9.3 (2) & -6.4 (35) & -4 (50) & -2.4 (59) & -2.7 (58) \\
  $n=2$ &
 -10 (0) & -7.6 (23) & -3.7 (73) & -2.3 (82) & -1.4 (86) & -2.7 (81) \\
   {\sc phc}, $n=1$ &
 -8.5 (0) & -6.7 (19) & -3.1 (62) & -0.7 (74) & 1.8 (72) & 1.9 (74) \\
  $n=2$ & 
 -8.6 (0) & -4.8 (55) & -1.8 (88) & -0.7 (92) & 0 (94) & -1 (91) \\
   \hline
\end{tabular}
  
\caption{\label{tab:delete} Bias and empirical gains in percent between brackets, see~\eqref{eq:gain}, for the standard and median based estimators for different values of $k_n$. The empirical results are based on 1000 replications generated on $[-n,n]^2$ for $n=1,2$ for the models {\sc poisson, lgcp, phc} where $5\%$ or $10\%$ of points are deleted to each configuration. This corresponds to the case (C) described in details above.}
\end{table}

Finally, we compared our estimator with the one proposed by \citet{berndt:stoyan:97}. In particular, we used the version presented  by \citet[p.~252]{illian:et:al:08}. 
Let $G$ be a grid of $\#G$ dummy points, and for any $g\in G$, let $a(g,\bX)$ be the Voronoi cell area of the cell corresponding to the closest point of $g$ in $\bX$, then $(\#G)^{-1}\sum_g a(g,\bX)^{-1}$ is (up to edge effects) an unbiased estimator of $\lambda$. \citet{illian:et:al:08}, proposed to replace the previous sample mean by  a sample median or sample trimmed-mean.
We investigated the latter estimator, referred to as the Voronoi estimator in the following, in a shorter simulation study. We considered only the {\sc poisson} model (similar results were observed for the {\sc lgcp} and {\sc phc} models), fixed the grid $G$ to a regular grid of $\#G=200^2$ dummy points (the results were very stable to that parameter) and used a symmetrically trimmed mean  with a fraction of $f=0.025,0.05$ or $0.1$ observations trimmed from each end.
The \texttt{spatstat R} package was used to compute the Voronoi tessellation and cell areas. As suggested by the authors, to correct border effects, we removed all border cells from the analysis. Empirical means and standard deviations for the Voronoi estimator in the settings (A)-(C), described above, are reported in Table~\ref{tab:stoyan}. 

In the setting (A), the Voronoi estimator has a  bias comparable to the one of the median-based estimator (see Table~\ref{tab:nothing}) when $f=0.025$ and $f=0.05$. The bias is surprisingly very large when $f=0.1$. When  extra points  are observed (setting (B)) or omitted (setting (C)), the results are less ambiguous: the Voronoi estimator remains strongly biased, even more than with the standard estimator $\widehat\lambda^{\mathrm{std}}$ in a few cases when $f=0.025$ or $f=0.1$. The choice $f=0.05$ seems to offer a better compromise. We computed the gains for the Voronoi estimator as we did for the median-based estimator in Tables~\ref{tab:add}-\ref{tab:delete} using~\eqref{eq:gain}. The results, not reported, show that the median-based estimator is more efficient in the settings (A), (B) and (C) when $f=0.025$ and $f=0.1$ and slightly more efficient when $f=0.05$.

From a computational point of view, the Voronoi estimator is more expensive to evaluate. For instance it takes  6 seconds in average to evaluate the Voronoi estimator when $n=2$ while it takes approximately 0.03 second to evaluate the median-based estimator (for all the values of $k_n=9,16,25,36,49$). This precludes from using the Voronoi estimator (at least in that form) for very large point pattern. Asymptotic properties were not the focus of the paper by \citet{berndt:stoyan:97} or the book by~\citet{illian:et:al:08}. But another argument in support of our approach is that we believe that deriving asymptotic results for the Voronoi estimator (consistency, asymptotic variance, central limit theorem) looks more awkward.\\


\begin{table}[htbp]
\centering 
\begin{tabular}{rrrrrr}
\hline
&\multicolumn{1}{c}{(A)}&\multicolumn{2}{c}{(B)}& \multicolumn{2}{c}{(C)} \\
&&  $\rho=0.05$ & $\rho=0.1$ & $\rho=0.05$ &  $\rho=0.1$ \\
\hline
\multicolumn{1}{l}{$f=0.025$ \qquad\qquad}&&&&&\\
 $n=1$ &
98.4   (5.8)&104.7 (5.7)& 109.7 (5.9)&96 (5.4)&94.9 (5.7)\\
$n=2$ &
98  (2.6)& 104.6 (2.6)&107.3 (2.8) &95.4 (3)&91.1 (3.1)\\
\multicolumn{1}{l}{$f=0.05$ \qquad\qquad}&&&&&\\
 $n=1$ &
99.3 (5.5)&103.4 (5.5)&104.9 (5.7)&94.4  (4.9)&92 (5.4)\\
$n=2$ &
98.8 (2.7)&101.8 (2.5)&103.8 (2.7)&94 (2.5)&91.9  (3)\\
\multicolumn{1}{l}{$f=0.1$ \qquad\qquad}&&&&&\\
$n=1$ & 
94.1 (5.1) & 97.1 (5)&98.5 (5.1)&91.5 (5.1)&88.9 (5.5)\\
$n=2$ & 
94.5 (2.3)&97.4 (2.4) &97.1 (2.7)&88.9 (2.3) &87.9 (2.7)\\
\hline
\end{tabular}
\caption{\label{tab:stoyan} Empirical means and standard deviations between brackets of estimates of the intensity $\lambda=100$ based on 1000 replications of the {\sc poisson} on $[-n,n]^2$ for $n=1,2$ and in the settings (A), (B) and (C) described in the text. The estimator considered in this table is the Voronoi cell estimator proposed by \citet[p.252]{illian:et:al:08}. 
}
\end{table}

\section{Conclusion} \label{sec:conclusion}

In this paper, we propose a median-based estimator of the intensity parameter of a stationary spatial point process. We prove asymptotic properties of this estimator as the observation window expands to $\R^d$. In particular for a large class of models, we show that the estimator $\widehat \lambda^J$ satisfies a central limit theory, which allows us to derive asymptotic confidence intervals. 

As a general conclusion of the simulation study, it turns out that the estimator $\widehat\lambda^J$ confirms expected asymptotic properties and improves the robustness property of the standard procedure. Even if the choice of the tuning parameter $k_n$  has a moderate influence on the empirical results when the observation window is large or when the point pattern is strongly contaminated,  it is an open question to propose a data-driven procedure to select the number of cells $k_n$.

In this paper, we did not aim at detecting outliers or detecting areas where problems are suspected (abundance or lack of points). If the assumption of stationarity seems valid, a large difference between the median-based estimator and a classical estimator of the intensity parameter might allow the user to reconsider the observation window in a second step. 

The research contained in this paper leads to interesting open issues: (i) It could be worth continuing the comparison between the Voronoi estimator and our approach. This would require to propose a data-driven procedure to fix the tuning parameter of the fraction of deleted observations, to investigate a more evolved border correction and to derive asymptotic properties. (ii) In the setting (B), we considered outliers as extra points added in a small subsquare. Extra points could be uniformly distributed on the observation domain. Such a situation would be closer to the application in image analysis investigated by \citet{berndt:stoyan:97}. A similar problem was also considered by \citet{redenbach:15}. (iii) A quantile based estimator is the most natural way of defining a robust estimator. A more advanced technique would consist in extending  our theory to M-estimators. (iv) Another step could be to tackle the problem of robust estimators for second-order characteristics  like the $K, F$ or $G$ functions. (v) Finally, extending the methodology and results to the estimation of the intensity of inhomogeneous spatial point processes constitutes also an interesting perspective.

\section*{Acknowledgements}
The author would like to thank sincerely Michaela Prokešová for discussing the initial idea of this paper, Frédéric Lavancier for discussions on an earlier version of the manuscript and Jérôme Lelong for a careful reading. The author is very grateful to  the associate editor and the referee in particular for pointing out interesting references, which enriched the simulation study. The research of the author is partially funded by Persyval-lab EA Oculo-Nimbus.

\appendix
\section{Proofs} \label{sec:proofs}

In all the proofs, $\kappa$ denotes a generic constant which may vary from line to line. For $k=(k_1,\dots,k_d)^\top\in \Z^d$, we denote  $|k|=\max(|k_1|,\dots,|k_d|)$.


\subsection{Auxiliary results}

We present in this section an auxiliary lemma providing a control of the covariance of counting variables and a general central limit theorem adapted to our context. 

\begin{lemma}\label{lem:cov}
We define $\mathcal C_\tau$ the cube centered at $0$ with volume $\tau^d c_n$, i.e.
\[
  \mathcal C_\tau=\left\{
u=(u_1,\dots,u_d)^\top \in\R^d: |u_l|\leq \tau c_n^{1/d}/2, l=1,\dots,d
\right\}.
\]
Under  Assumption (iii), we have the  following three statements.\\
(a) For any $\tau\in (0,1]$
\[
  \Var(N({\mathcal C_\tau})) \sim |\mathcal C_\tau| \left( \lambda + \lambda^2 \int_{\R^d} (g(w)-1)\dd w\right)
\]
as $n\to \infty$.\\\noindent
(b) Let $\varepsilon\in(0,1)$ then
\[
  \Cov(N({\mathcal C_{1-\varepsilon}}),N({\mathcal C_1})) \sim
\lambda |\mathcal C_{1-\varepsilon}| + \lambda^2|\mathcal C_{1-\varepsilon/2}|\int_{\R^d} (g(w)-1)\dd w
\]
as $n\to \infty$.\\\noindent
(c) Let $(\varepsilon_n)_{n\geq 1}$ be a sequence of real numbers such that $\varepsilon_n \to 0$ and $c_n^{1/d}\varepsilon_n\to \infty$ as $n\to \infty$, then 
\begin{align*}
\Var( N({\mathcal C_{1-\varepsilon_n}})) &\sim  \Var( N({\mathcal C_{1}})) \sim \Cov(N({\mathcal C_{1-\varepsilon_n}}),N({\mathcal C_{1}}) )\\ &\sim c_n \left( \lambda + \lambda^2 \int_{\R^d}(g(w)-1)\dd w\right)
\end{align*}
as $n\to \infty$.
\end{lemma}

\begin{proof}
(a) is a classical result, see e.g. \citet{heinrich:prokevsova:10}. As we need to refer to specific equations, we report the proof here. By Campbell's Theorem and since $\bX$ admits a pair correlation function
\begin{align}
  \Var(N({\mathcal C_\tau})) &= \lambda |\mathcal C_\tau| + \int_{\R^d}\int_{\R^d} \1(u\in \mathcal C_\tau)\1(v\in \mathcal C_\tau)(g(u-v)-1)\dd u \dd v \nonumber \\
  &=\lambda |\mathcal C_\tau| + \lambda^2\int_{\R^d} |\mathcal C_\tau \cap (\mathcal C_\tau)_{-w}| (g(w)-1)\dd w \nonumber\\
  &=\lambda |\mathcal C_\tau| + \lambda^2\int_{\mathcal C_{2\tau}} |\mathcal C_\tau \cap (\mathcal C_\tau)_{-w}| (g(w)-1)\dd w \label{eq:cov1}\\
  &= \lambda |\mathcal C_\tau| + \lambda^2\int_{\mathcal C_{2\tau}} \prod_{l=1}^d (\tau c_n^{1/d}-|w_l|)
  (g((w_1,\dots,w_d)^\top)-1)\dd w_1 \dots \dd w_d \label{eq:cov2}\\
  &\sim |\mathcal C_\tau| \left( \lambda + \lambda^2 \int_{\R^d} (g(w)-1)\dd w\right) \nonumber
\end{align} 
by Assumption (iii).\\
(b) For brevity, let $K_\varepsilon$ denote the covariance to evaluate. Following (a) we have
\begin{align*}
K_\varepsilon &=\lambda |\mathcal C_{1-\varepsilon} \cap \mathcal C_1 | + \int_{\R^d}\int_{\R^d}\1(u\in \mathcal C_{1-\varepsilon})\1(v\in \mathcal C_1)(g(u-v)-1)\dd u \dd v \\
&= \lambda |\mathcal C_{1-\varepsilon}| + \lambda^2\int_{\R^d} |\mathcal C_{1-\varepsilon} \cap (\mathcal C_1)_{-w}| (g(w)-1)\dd w. 
\end{align*}
Let $w=(w_1,\dots,w_d)^\top$. We can check that 
\[
  |\mathcal C_{1-\varepsilon} \cap (\mathcal C_1)_{-w}| = \left\{
\begin{array}{ll}
  0 & \mbox{ if } w \in \R^d\setminus \mathcal C_{2-\varepsilon} \\
  \prod_{l=1}^d \left( \big(1-\frac\varepsilon 2\big) c_n^{1/d} -|w_l|\right) & \mbox{ if } w \in \mathcal C_{2-\varepsilon}
\end{array}
  \right.
\]
whereby we deduce using \eqref{eq:cov1}-\eqref{eq:cov2} and Assumption~(iii) that
\begin{align*}
K_\varepsilon &= \lambda |\mathcal C_{1-\varepsilon}| + \lambda^2\int_{\mathcal C_{2-\varepsilon}} \prod_{l=1}^d \Big((1-\varepsilon/2) c_n^{1/d}-|w_l|\Big)
  \Big(g((w_1,\dots,w_d)^\top)-1\Big)\dd w_1 \dots \dd w_d \\
  &=\lambda |\mathcal C_{1-\varepsilon}| + \lambda^2\int_{\mathcal C_{2-\varepsilon}} |\mathcal C_{1-\varepsilon/2} \cap (C_{1-\varepsilon/2})_{-w}| (g(w)-1)\dd w\\
  &\sim \lambda |\mathcal C_{1-\varepsilon}| + \lambda^2|\mathcal C_{1-\varepsilon/2}|\int_{\R^d} (g(w)-1)\dd w
\end{align*}
as $n \to \infty$.\\
(c) The assumptions on the sequence $(\varepsilon_n)$ allow us to apply (a)-(b) which leads to the result since $|\mathcal C_1|\sim |\mathcal C_{1-\varepsilon_n}|\sim |\mathcal C_{1-\varepsilon_n/2}|\sim c_n$ as $n\to \infty$.
\end{proof}

Now we present a central limit theorem for stationary random fields with asymptotic covariance matrix  not necessarily positive definite. It is  very close to~\citet[Theorem 3.3.1]{guyon:91} and to~\citet[Theorem 1]{karaczony:06} but we were not able to find it in the following form in the literature.

For two square matrices $A,B$, $A\geq B$ (resp. $A>B$) means that $A-B$ is a positive (resp. positive definite) matrix. Finally, $\|A\|$ stands for the Frobenius norm of $A$  given by $\|A\|=\mathrm{Tr}(A^\top A)^{1/2}$.

\begin{theorem} \label{thm:degenere}
Let $(X_k,k\in\Z^d)$ be a stationary  random field in a measurable space $S$. Let $\mathcal K_n\subset \Z^d$ with $k_n=|\mathcal K_n| \to \infty$ as $n\to \infty$. For any $n\geq 1$ and  $k\in \mathcal K_n$, we define $Y_{n,k}=f_{n,k}(X_k)$ where $f_{n,k}:S \to \R^p$ for some $p\geq 1$ is a measurable function. We denote by $S_n= \sum_{k\in \mathcal K_n} Y_{n,k}$ and by $\Sigma_n=\Var(S_n)$ and assume that for any $n\geq 1$, $k\in \mathcal K_n$, $\E Y_{n,k}=0$. We also assume that \\
(I) $\sup_{n\geq 1}\sup_{k\in \mathcal K_n} \|Y_{n,k}\|_\infty<\infty$. \\
(II)There exists $\eta>0$ such that $\alpha_{2,\infty}(m)=\mathcal O(m^{-d(1+\eta)})$.\\
(III) There exists $\Sigma\geq 0$ a $(p,p)$ matrix with rank $1\leq r\leq p$
such that $k_n^{-1} \Sigma_n \to \Sigma$ as $n\to \infty$.\\
Then, $k_n^{-1/2} S_n \to \mathcal N(0,\Sigma)$ in distribution as $n\to \infty$.
\end{theorem}
We present Theorem~\ref{thm:degenere} for bounded random vectors and with only one mixing coefficient, namely $\alpha_{2,\infty}$. It can obviously be generalized along similar lines as in~\citet[Theorem 3.3.1]{guyon:91}.

\begin{proof}
Assume $\Sigma>0$, then for $n$ large enough $k_n^{-1} \Sigma_n \geq \Sigma/2>0$, which combined with Assumptions (I)-(II) allows us to apply~\citet[Theorem 1]{karaczony:06} to conclude the result.

The end of the proof follows the same arguments as the proof of a central limit theorem for triangular arrays of conditionally centered random fields obtained by~\citet[Theorem 2]{coeurjolly:lavancier:13}. If $\Sigma$ is not positive definite, we can find an orthonormal basis  $(h_1,\dots,h_p)$ of $\R^p$ where the $f_i$'s are eigenvectors of $\Sigma$. We let $(f_1,\dots,f_r)$ be the basis of the image of $\Sigma$ and $(f_{r+1},\dots,f_p)$ be the basis of its kernel. Let also $H_{Im}$ (resp. $H_{Ker}$) be the matrix formed by the column vectors of $(f_1,\dots,f_r)$ (resp. $(f_{r+1},\dots,f_p)$). Similarly for $v\in \R^p$, we denote by $v_j$ its $j$th coordinate in the basis of $(f_1,\dots,f_p)$, $v_{Im}=(v_1,\dots,v_r)$ and $v_{Ker}=(v_{r+1},\dots,v_p)$. Using the Cramer-Wold device, we need to prove that for any $v\in \R^p$, $v^\top k_n^{-1/2} S_n$ converges towards a Gaussian random variable. We have
\[
  v^\top k_n^{-1/2} S_n  = v_{Im}^\top H_{Im}^\top k_n^{-1/2}S_n
 + v_{Ker}^\top H_{Ker}^\top k_n^{-1/2}S_n.
 \]
Let $S_n^\prime=\sum_k Y_{n,k}^\prime$ where $Y_{n,k}^\prime=H_{Im}^\top Y_{n,k}$. The random variables $Y_{n,k}^\prime$ are bounded variables for any $n\geq 1$ and  $k\in \mathcal K_n$. By assumption (III), $k_n^{-1}\Var(S_n^\prime)\to H_{Im}^\top \Sigma H_{Im}$ which is a positive definite matrix since $r\geq 1$. Therefore from the first part of the proof, $v_{Im}^\top H_{Im}^\top k_n^{-1/2}S_n$ tends to a Gaussian random variable in distribution as $n\to \infty$. By Slutsky's Lemma (see e.g. \citet{van:00}), the proof will be done if $v_{Ker}^\top H_{Ker}^\top k_n^{-1/2}S_n$ tends to 0 in probability as $n\to \infty$. Since, $ H_{Ker}^\top \Sigma H_{Ker} =0$, the expected convergence follows from
\begin{align*}
\Var (v_{Ker}^\top H_{Ker}^\top k_n^{-1/2}S_n) &= v_{Ker}^\top H_{Ker}^\top k_n^{-1} \Sigma_n H_{Ker} v_{Ker} \\
&= v_{Ker}^\top H_{Ker}^\top (k_n^{-1} \Sigma_n-\Sigma) H_{Ker} v_{Ker} \\
&\leq \|v_{Ker}\| \,\|H_{Ker}\|\, \|k_n^{-1} \Sigma_n -\Sigma\| 
\end{align*}
which tends to 0 by Assumption~(III).
\end{proof}

\subsection{Proof of Proposition~\ref{prop:iiCox}} \label{sec:proofiiCox}

\begin{proof}
{\it Assumption (ii-1)} corresponds to Proposition~\ref{prop:MeZcheckCox}.\\

{\it Assumptions (ii-2) and (ii-3)}. 
By definition of $\bX$,
\[
  \sqrt{2\pi \lambda c_n} \P(N(C_{n,0})=\lfloor t_n \rfloor \mid \xi) = \frac{\left(\int_{C_{n,0}} \xi(s)\dd s\right)^{\lfloor t_n \rfloor}e^{-\int_{C_{n,0}}\xi(s) \dd s}}{\lfloor t_n \rfloor^{\lfloor t_n \rfloor} e^{-\lfloor t_n \rfloor}} \;v_n
 \]
 where 
 \[
   v_n= \sqrt{\frac{\lambda c_n}{\lfloor t_n \rfloor}} \frac{\sqrt{2\pi}\lfloor t_n \rfloor ^{\lfloor t_n \rfloor +1/2} e^{-\lfloor t_n \rfloor}}{\lfloor t_n \rfloor !}.
 \]
Since $t_n/(\lambda c_n)\to 1$ as $n\to \infty$, then using Stirling's Formula we obviously have $v_n\to 1$ as $n\to \infty$. Now using the notation $T_n=\lfloor t_n \rfloor^{-1}\int_{C_{n,0}}\xi(s)\dd s$, we rewrite the first equation as follows
\[
(v_n)^{-1} \sqrt{2\pi \lambda c_n} \P(N(C_{n,0})=\lfloor t_n \rfloor \mid \xi) = T_n^{\lfloor t_n \rfloor} e^{\lfloor t_n \rfloor (1-T_n)} = A_n B_n  
\]
where $A_n$ and $B_n$ are defined by
\[
A_n = e^{-{\lfloor t_n \rfloor}(T_n-1)^2/2 } \quad \mbox{ and }
B_n = e^{\lfloor t_n \rfloor \left( \log T_n - (T_n-1)+(T_n-1)^2/2 \right)  }.
\]
Let $\eta$ be given by Assumption~(iv).
Since $\E |\xi(0)|^{2+\delta} <\infty$ for some $\delta>2/\eta$, we are ensured that $\alpha_{2,\infty}= \mathcal O(m^{-\nu})$ for some $\nu>d(2+\delta)/\delta$. Therefore, we can apply~\citet[Theorem 3.3.1]{guyon:91} and show that there exists $\tau>0$ such that $\sqrt{\lambda c_n} (I_n-1)\to \mathcal N(0,\tau^2)$ in distribution where $I_n=(\lambda c_n)^{-1} \int_{C_{n,0}}\xi(s)\dd s$. To compute $\tau^2$, we observe that using the definition of a Cox point process
\begin{align*}
\Var(N(C_{n,0})) &= \E \left( \Var(N({C_{n,0}}) \mid \xi)\right) + \Var\left( \E \left( N({C_{n,0}})\right) \mid \xi \right) \\
&= \lambda c_n + \Var\int_{C_{n,0}} \xi(s)\dd s.
\end{align*}
We use Assumption (iii) and Lemma~\ref{lem:cov} (a) to deduce that as $n\to \infty$
\[
  \Var\int_{C_{n,0}} \xi(s)\dd s  \sim \lambda^2 c_n\int_{\R^d}(g(w)-1)\dd w  
\]
which leads to $\Var( \sqrt{\lambda c_n} I_n) \to  \tau^2=\lambda \int_{\R^d}(g(w)-1)\dd w$ as $n\to \infty$. From the definition of $t_n$ and  Slutsky's Lemma, it can be shown that $\sqrt{\lfloor t_n \rfloor} (T_n-I_n)\to 0$ in probability which leads to $T_n \to 1$ in probability and $\sqrt{\lfloor t_n \rfloor}(T_n-1)\to\mathcal N(0,\tau^2)$ in distribution. We deduce that $A_n \to 
A=e^{-\tau^2 L^2/2}$ in distribution, where $L\sim \mathcal N(0,1)$, which, by the uniform integrability of the sequence $(A_n)_n$,  leads to $A_n \to A$ in $L^1$. Now a Taylor expansion shows that there exists $\widetilde T_n \in (0\wedge (T_n-1), 0\vee (T_n-1))$ such that
\[
  |\log(B_n)| = \lfloor t_n \rfloor |T_n-1| \, 
  \frac{\widetilde T_n^2}{1+\widetilde T_n} \leq \lfloor t_n \rfloor (T_n-1)^2 \frac{|T_n-1|}{\widetilde T_n+1}.
\]
It is clear that $\widetilde T_n$ tends to 0 in probability as $n\to \infty$, which yields that $\log(B_n) \to 0$ and $B_n\to 1$ in probability by Slutsky's Lemma. Again, the uniform integrability assumption of the sequence $(B_n)_n$ implies that $B_n\to 1$ in $L^1$. Since 
\[
  |A_n B_n -A| \leq A_n |B_n-1| + |A_n-A| \leq |B_n-1| + |A_n-A|
\] we conclude that $A_n B_n \to A$ in $L^1$ as $n\to \infty$. In other words as $n\to \infty$
\[
  \sqrt{2\pi \lambda c_n} \P(N(C_{n,0})=\lfloor t_n\rfloor ) \sim v_n^{-1} \E \left(
\sqrt{2\pi \lambda c_n} \P(N(C_{n,0})=\lfloor t_n\rfloor \mid \xi ) 
   \right) \to \E(A).
\]
Using the definition of the moment generating function of a $\chi_1^2$ distribution, we have $\E(A) = (1+\tau^2)^{-1/2}$ whereby we deduce that
\[
  \sqrt{c_n} \P(N(C_{n,0})=\lfloor t_n\rfloor ) \to \left( 2\pi \lambda(1+\tau^2)\right)^{-1/2} = \left( 2\pi \sigma^2 \right)^{-1/2}
\]
with $\sigma^2=\lambda + \lambda^2 \int_{\R^d}(g(w)-1)\dd w$.
\end{proof}

\subsection{Proof of Theorem~\ref{thm:cltF}}

\begin{proof}
We focus only on (a) as (b) follows from (a), Slutsky's Lemma and Assumption (ii-2). Let $t_n=\lambda c_n+a_n$. By definition
\begin{equation}
  \label{eq:sum}
  \Fe(t_n ;\bZ) -F_Z(t_n) = \frac1{k_n} \sum_{k\in \mathcal K_n} \big(\1(Z_{n,k}\leq t_n) - \P(Z_{n,k}\leq t_n) \big).    
\end{equation}
Let $(\varepsilon_n)_{n\geq 1}$ be a sequence of real numbers such that $\varepsilon_n\to 0$ and $\varepsilon_n c_n^{1/d}\to \infty$ as $n\to \infty$.We denote by $Z_{n,k}^-=N( C_{n,k}^-)+U_k$ where $C_{n,k}^-$ is the erosion of the cell $C_{n,k}$ by a closed ball with radius $\varepsilon_n c_n^{1/d}$. Two cells $C_{n,k}^-$ and $C_{n,k^\prime}^-$ for $k,k^\prime\in \mathcal K_n$ ($k\neq k^\prime$) are therefore at distance greater than $2\varepsilon_n c_n^{1/d}$. To prove Theorem~\ref{thm:cltF} (a), we use the blocking technique introduced by~\citet{ibragimov:linnik:71} and applied to spatial point processes by \cite{guan:loh:07,guan:sherman:calvin:07} and \citet{prokevsova:jensen:13}.
To this end, we need additional notation. For any $n\geq 1$ and  $k\in \mathcal K_n$, let $t_n^-=\lambda |C_{n,k}^-|+1/2= \lambda (1-\varepsilon_n)^d c_n +1/2$ and let $(\widetilde Z_{n,k}^-, k\in \mathcal K_n)$ be a collection of independent random variables such that $\widetilde Z_{n,k}^- \stackrel{d}{=}Z_{n,k}^-$. We decompose the sum in~\eqref{eq:sum} as follows
\begin{equation}\label{eq:sum2}
  \sum_{k\in \mathcal K_n} \big(\1(Z_{n,k}\leq t_n) - \P(Z_{n,k}\leq t_n) \big) = D_n + S_n^- + \widetilde S_n^-
\end{equation}
where
\begin{align*}
D_n &= \sum_{k \in \mathcal K_n} D_{n,k} =
\sum_{k \in \mathcal K_n} \big\{\1(Z_{n,k}\leq t_n) - \P(Z_{n,k}\leq t_n) - \1(Z_{n,k}^-\leq t_n^-)+\P(Z_{n,k}^-\leq t_n^-) \big\}\\
S_n^- &= \sum_{k\in \mathcal K_n} \1(Z_{n,k}^-\leq  t_n^-)-\P(Z_{n,k}^-\leq t_n^-)
\\
  \widetilde S_n^- &= \sum_{k\in \mathcal K_n} \1(\widetilde Z_{n,k}^-\leq t_n^-)-\P(\widetilde Z_{n,k}^-\leq t_n^-).
\end{align*}
 We split the proof into three steps. As $n\to \infty$, we prove that

\noindent{\it Step 1.} $ D_n / \sqrt{k_n} \to 0$ in probability.

\noindent{\it Step 2.}  for any $u \in \R$, $\phi_n^-(u)-\widetilde \phi_n^-(u)\to 0$ as $n\to \infty$ where $i=\sqrt{-1}$, $\phi_n^-(u) =\E(e^{i uS_n^-/\sqrt{k_n}})$
and $\widetilde\phi_n^-(u) =\E(e^{i u \widetilde S_n^-/\sqrt{k_n}})$, which will imply that
$(S_n^- - \widetilde{S}_n^- ) /\sqrt{k_n} \to 0$.

\noindent{\it Step 3.} $\widetilde S_{n}^- /\sqrt{k_n}\to \mathcal N(0,1/4)$ in distribution.\\

The conclusion will follow directly from Steps 1-3, \eqref{eq:sum}-\eqref{eq:sum2} and Slutsky's Lemma.\\

\noindent{\it Step 1.}  To achieve this step, we will prove that $k_n^{-1} \E(D_n^2)=k_n^{-1}\Var(D_n)\to 0$ as $n \to \infty$. We have
\[
  \frac1{k_n} \Var(D_n) = \frac1{k_n} \sum_{\begin{subarray}{c}
k,k^\prime \in \mathcal K_n \\ |k-k^\prime| \leq 1    
  \end{subarray}
  } \Cov(D_{n,k},D_{n,k^\prime}) +
  \frac1{k_n} \sum_{\begin{subarray}{c}
k,k^\prime \in \mathcal K_n \\ |k-k^\prime| > 1    
  \end{subarray}
  } \Cov(D_{n,k},D_{n,k^\prime}) .
\]
Let $k,k^\prime\in \mathcal K_n$  with $k\neq k^\prime$, Assumption~(i) asserts that $d(C_{n,k},C_{n,k^\prime})  = |k-k^\prime-1| c_n^{1/d}$. Since $D_{n,k}\in \mathcal F(C_{n,k})$ and $D_{n,k^\prime}\in \mathcal F(C_{n,k^\prime})$, we have from \citet[Lemma 2.1]{zhengyan:96}
\begin{align*}
  \Cov(D_{n,k},D_{n,k^\prime}) &\leq 4 \alpha_{c_n,c_n} (|k-k^\prime-1|c_n^{1/d}) \\
  &\leq 4c_n \alpha(|k-k^\prime-1|c_n^{1/d}) = \mathcal O(|k-k^\prime-1|^{-d(1+\eta)} c_n^{-\eta}).
\end{align*}
Since the series $\sum_{k\in \Z^d\setminus\{0\}} |k|^{-d(1+\eta)}$ converges, it is clear that
\begin{equation}\label{eq:covRight}
   \frac1{k_n} \sum_{\begin{subarray}{c}
k,k^\prime \in \mathcal K_n \\ |k-k^\prime| > 1    
  \end{subarray}
  } \Cov(D_{n,k},D_{n,k^\prime}) = \mathcal O(c_n^{-\eta}),
\end{equation}
which tends to 0 as $n\to \infty$. Since the variables $D_{n,k}$ are identically distributed, we get from the Cauchy-Schwarz inequality 
\begin{align*}
\Big| \frac1{k_n} \sum_{\begin{subarray}{c}
k,k^\prime \in \mathcal K_n \\ |k-k^\prime| \leq 1    
  \end{subarray}
  } \Cov(D_{n,k},D_{n,k^\prime})\Big| &\leq 
  \frac1{k_n} \sum_{\begin{subarray}{c}
k,k^\prime \in \mathcal K_n \\ |k-k^\prime| \leq 1    
  \end{subarray}
  } \sqrt{ \Var(D_{n,k}) \Var(D_{n,k^\prime})} \\
  &\leq \Var(D_{n,0}) \frac1{k_n} \sum_{\begin{subarray}{c}
k,k^\prime \in \mathcal K_n \\ |k-k^\prime| \leq 1    
  \end{subarray}
  }1 \\ &\leq 3^d \Var(D_{n,0}).
\end{align*}
Thus, Step 1 is achieved once we prove that $\Var (D_{n,0}) \to 0$ as $n\to \infty$.
A straightforward calculation yields that
\begin{align*}
\Var(D_{n,0})&= \P(Z_{n,0}\leq t_n)(1-\P(Z_{n,0}\leq t_n)) +
\P(Z_{n,0}^-\leq t_n^-)(1-\P(Z_{n,0}^-\leq t_n^-)) \\
&\qquad + 2\P (Z_{n,0}\leq t_n) \P(Z_{n,0}^-\leq t_n^-) - 2 \P\big( Z_{n,0}\leq t_n, Z_{n,0}^-\leq t_n^-\big).
\end{align*}
Let $\Delta_j$ be the unit cube centered at $j\in \Z^d$ and let $\mathcal J_n=\{j\in \Z^d: \Delta_j\cap C_{n,0} \neq \emptyset\}$. We denote by $Y_{n,j}$ the  random vector
\[
  Y_{n,j} = \left( \frac{U_{0}}{j_n} + \1(u\in  C_{n,0}\cap \Delta_j) , \frac{U_{0}}{j_n} + \1(u\in  C_{n,0}^-\cap \Delta_j) \right)^\top
\]
where $j_n=|\mathcal J_n|$ satisfies $j_n\sim c_n$ as $n\to\infty$. We have $(Z_{n,0},Z_{n,0}^-)^\top =\sum_{j\in \mathcal J_n}Y_{n,j}$ and we note that $\sup_{n\geq 1}\sup_{j\in \mathcal J_n} \|Y_{n,j}\|_\infty<\infty$.
Since $\varepsilon_n\to 0$ and $c_n^{1/d}\varepsilon_n\to \infty$ as $n\to \infty$, we can apply Lemma~\ref{lem:cov} (c) to derive 
\[
  \Var(Z_{n,0}) \sim \Var(Z_{n,0}^-) \sim \Cov(Z_{n,0},Z_{n,0}^-)\sim \frac1{12} + \sigma^2 c_n
\]
where $\sigma^2=\lambda +\lambda^2\int_{\R^d} (g(w)-1)\dd w$. In other words,
\begin{equation}\label{eq:convZn0}
  j_n^{-1}\Var( (Z_{n,0},Z_{n,0}^-)^\top) \to \Sigma = \sigma^2 \left( 
\begin{array}{cc} 
1&1\\1&1  
\end{array}
  \right),
\end{equation}
which is a matrix with rank 1. By combining this with Assumption~(iv), we can apply Theorem~\ref{thm:degenere} to get as $n\to \infty$
\[
  c_n^{-1/2} \left(Z_{n,0}-\E(Z_{n,0}), Z_{n,0}^- -\E(Z_{n,0}^-) \right)^\top \to \mathcal N(0,\Sigma)
\]
in distribution. Since $t_n^-=\E(Z_{n,0}^-)$ and $\E(Z_{n,0})-t_n=1/2-a_n=o(c_n^{1/2})$ by definition of $t_n$, an application of Slutsky's Lemma yields that
\[
  c_n^{-1/2} \left(Z_{n,0}-t_n, Z_{n,0}^- -t_n^-) \right)^\top \to \mathcal N(0,\Sigma)
\]
in distribution as $n\to\infty$ whereby we deduce that 
\begin{equation}\label{eq:convtn}
  \P( Z_{n,0}\leq t_n) \to 1/2
\quad \mbox{ and } \quad
  \P(Z_{n,0}^-\leq t_n^- ) \to 1/2.
\end{equation}
\citet{rose:smith:96} proved that if $U=(U_1,U_2)^\top$ follows a bivariate normal distribution with mean 0, variance 1 and  correlation $\rho$, $\P(U_1\leq 0,U_2\leq 0)=1/4+\sin^{-1}(\rho)/2\pi$ which equals to 1/2 when $\rho=1$. 
From~\eqref{eq:convZn0}, this shows that $\P( Z_{n,0}\leq t_n, Z_{n,0}^-\leq t_n^-) \to 1/2$ as $n\to \infty$. As a consequence, $\Var(D_{n,0})\to 0$ which combined with~\eqref{eq:covRight} leads to $k_n^{-1}\Var(D_n)\to 0$ as $n\to \infty$.\\

\noindent{\it Step 2.} This step is the core of the blocking technique.  
Let $h$ denote a bijection from $\mathcal K_n$  to $\{1,\dots,k_n\}$.
Let $j\in \{1,\dots,k_n\}$ and $V_j=e^{i u(\1(Z_{n,h^{-1}(j)}\leq t_n^-) -\P(Z_{n,h^{-1}(j)}\leq t_n^-) )/\sqrt{k_n}}$. Then
\[
  \phi_n^-(u) = \E  \prod_{j=1}^{k_n} V_j
  \quad \mbox{ and } \quad 
  \widetilde \phi_n^-(u) = \prod_{j=1}^{k_n} \E(V_j).
\] 
and 
\[
  |\phi_n^-(u) -\widetilde \phi_n^-(u)| \leq \sum_{j=1}^{k_n-1} \big|
\E \big(\prod_{s=1}^{j+1} V_s \big)- \E \big(\prod_{s=1}^j V_s\big) \E(V_{j+1}) \big|.
\]
Let $j \in\{1,\dots,k_n-1\}$ and $A_j=\prod_{s=1}^j V_s$. Clearly, $A_j\in \mathcal F(\cup_{s=1}^j C^-_{n,h^{-1}(s)})$ and $V_{j+1} \in \mathcal F(C^-_{n,h^{-1}(j+1)})$, $|\cup_{s=1}^j C^-_{n,h^{-1}(s)}|=j (1-\varepsilon_n)^d c_n$, $|C^-_{n,h^{-1}(j+1)}|=(1-\varepsilon_n)^d c_n$ and $d(\cup_{s=1}^j C^-_{n,h^{-1}(s)},C^-_{n,h^{-1}(j+1)}) \geq 2 \varepsilon_n c_n^{1/d}$. Since $A_j$ and $V_{j+1}$ are bounded random variables, we have the following upper-bound on their covariance by means of the strong mixing coefficient, see \citet[Lemma 2.1]{zhengyan:96}
\begin{align*}
\Cov(A_j,V_{j+1}) &\leq 4 \alpha_{j(1-\varepsilon_n)^d c_n,(1-\varepsilon_n)^d c_n}(2\varepsilon_n c_n^{1/d}) \\
&\leq 4jc_n \sup_p \frac{\alpha_{p,p}(2\varepsilon_n c_n^{1/d})}p \\
&\leq 4 c_n k_n \mathcal O(\varepsilon_n^{-d(1+\eta)} c_n^{-(1+\eta)}) =\mathcal O(k_n \varepsilon_n^{-d(1+\eta)} c_n^{-\eta}) 
\end{align*}
whereby we deduce that $|\phi_n^-(u) -\widetilde \phi_n^-(u)|=\mathcal O(k_n^2 c_n^{-\eta} \varepsilon_n^{-d(1+\eta)})$. Now we can fix the sequence $(\varepsilon_n)_{n\geq 1}$. Specifically, we set $\varepsilon_n=c_n^{(\eta^\prime-\eta)/d(1+\eta)}$ for some $0<\eta^\prime<\eta$. This choice ensures that $\varepsilon_n\to 0$, $c_n^{1/d}\varepsilon_n=c_n^{(1+\eta^\prime)/d(1+\eta)}\to \infty$  and yields that
$
  |\phi_n^-(u) -\widetilde \phi_n^-(u)|=\mathcal O({k_n^2}/{c_n^{\eta^\prime}}  )
$
which tends to 0 as $n\to \infty$ by Assumption (i).\\

\noindent{\it Step 3.} Since $\widetilde Z_{n,k}^-\stackrel{d}{=} Z_{n,k}^-$ and since $\P(Z_{n,k}^- \leq t_n^-)\to 1/2$ as $n\to \infty$ from Step 2, we deduce that
\begin{align*}
\Var(\widetilde S_n^-) &= \sum_{k\in \mathcal K_n} \P(\widetilde Z_{n,k}^- \leq t_n^-) (1-\P(\widetilde Z_{n,k}^- \leq t_n^-)) \\
&= k_n\P(\widetilde Z_{n,0}^- \leq t_n^-) (1-\P(\widetilde Z_{n,0}^- \leq t_n^-))  \sim k_n/4
\end{align*}
as $n\to \infty$. Since $(\1(\widetilde Z_{n,k}^- \leq t_n^-), k\in \mathcal K_n)$ is a collection of bounded and independent random variables, Step 3 follows from an application of Lyapounov Theorem.
\end{proof}


\subsection{Proof of Theorem~\ref{thm:bahadur}}

\begin{proof}
(a) Let us define for any $t\geq 0$
\[
  A_n = \sqrt{\frac{k_n}{c_n}} \left( \Mee(\bZ)-\Me_Z\right) 
  \quad \mbox{ and } \quad
  B_n(t) = \sqrt{\frac{k_n}{c_n}} \left(
  \frac{F_Z(t)-\Fe(t;\bZ)}{f_Z(\Me_Z)}
  \right).
\]
We  have to prove that $A_n-B_n(\Me_Z)$
converges in probability to 0 as $n\to\infty$. The proof is based on
the application of \citet[Lemma~1]{ghosh:71} which consists in
satisfying  the two following conditions:\\
(I) for all $\delta>0$, there exists $\varepsilon=\varepsilon(\delta)$ such that $\P(|B_n(\Me_Z)|>\varepsilon)< \delta$.\\
(II) for all $y\in \R$ and $\varepsilon>0$
\[
\lim_{n\to \infty}\P(A_n\leq y, B_n(\Me_Z)\geq y+\varepsilon) =
\lim_{n\to \infty}\P(A_n\geq y+\varepsilon, B_n(\Me_Z)\leq y)=0.  
\]

In particular \noindent(I) is  fulfilled if we prove that $\Var
B_n(\Me_Z)=\mathcal{O}(1)$. The proof of Theorem~\ref{thm:cltF} shows that $\Var \Fe(\Me_Z ; \bZ) = \mathcal O(k_n^{-1})$ as $n\to \infty$. By Assumption~(ii-3), we obtain
\[
  \Var B_n(\Me_Z) = \frac{1}{c_n f_Z(\Me_Z)^2} \Var( \sqrt{k_n} \Fe(\Me_Z;\bZ)) = \mathcal O(1).
\]
(II) Let $y\in \R$ (and without loss of generality, assume $y\geq 0$). By definition of the sample median, we have
\begin{align*}
\left\{ A_n\leq y\right\} &= \left\{ \Mee(\bZ) \leq \Me_Z+y \sqrt{{c_n}/{k_n}} \right\} \\
&= \left\{ \frac12 \leq \Fe\left( \Me_Z+y \sqrt{{c_n}/{k_n}} \,\right)\right\} \\
&= \left\{ B_n \left( \Me_Z+y \sqrt{{c_n}/{k_n}}\right) \leq y_n\right\}
\end{align*}
where 
\[
  y_n = \sqrt{{k_n}/{c_n}} \,\frac1{f_Z(\Me_Z)} \left( F_Z\left(
 \Me_Z+y \sqrt{{c_n}/{k_n}}
 \right) - F_Z(\Me_Z) \right).
\]
Now, we  intend to prove that as $n\to \infty$, $y_n\to y$ and $\widetilde B_n=B_n(\Me_Z+y \sqrt{c_n/k_n})-B_n(\Me_Z)\to 0$ in probability. First, since $Z$ admits a density everywhere, there exists $\tau_n \in (\Me_Z,\Me_Z+y\sqrt{c_n/k_n})$ such that $y_n = y \; {f_Z(\tau_n)}/{f_Z(\Me_Z)}$.
From~\eqref{eq:fZ}
\begin{align*}
  \frac{f_Z(\tau_n)}{f_Z(\Me_Z)} &= \frac{\P(N(C_{n,0})=\lfloor \tau_n \rfloor)}{\P(N(C_{n,0})=\lfloor \Me_Z \rfloor)} ,
\end{align*}
which tends to 1 by Assumptions (ii-1)-(ii-2) and implies the convergence of $y_n$ towards $y$. Second, we show that $\Var(\widetilde B_n)\to 0$ as $n\to \infty$ by decomposing the variance as follows. Let $\widetilde B_{n,k} = \1(\Me_Z\leq Z_{n,k} \leq \Me_Z + y\sqrt{c_n/k_n}) - \P(\Me_Z\leq Z_{n,k} \leq \Me_Z + y\sqrt{c_n/k_n})$

\begin{align}
\Var(\widetilde B_n)&= \frac1{c_n f_Z(\Me_Z)^2} \, \frac1{k_n} \sum_{k,k^\prime\in \mathcal K_n} \Cov(\widetilde B_{n,k},\widetilde B_{n,k^\prime}) \nonumber\\
&\leq \frac{\kappa}{k_n} \sum_{\begin{subarray}{c}
k,k^\prime \in \mathcal K_n \\ |k-k^\prime|\leq 1
\end{subarray}} |\Cov(\widetilde B_{n,k},\widetilde B_{n,k^\prime})| +  \frac{\kappa}{k_n} \sum_{\begin{subarray}{c}
k,k^\prime \in \mathcal K_n \\ |k-k^\prime|> 1
\end{subarray}} |\Cov(\widetilde B_{n,k},\widetilde B_{n,k^\prime})| .\label{eq:covUn}
\end{align}
We follow the proof of Step 1 of Theorem~\ref{thm:cltF}. For any $k,k^\prime\in \mathcal K_n$ $k\neq k^\prime$, $\Cov(\widetilde B_{n,k},\widetilde B_{n,k^\prime})=\mathcal O(|k-k^\prime-1|^{-d(1+\eta)}c_n^{-\eta})$. So
\[
  \frac{1}{k_n} \sum_{\begin{subarray}{c}
k,k^\prime \in \mathcal K_n \\ |k-k^\prime|> 1
\end{subarray}} |\Cov(\widetilde B_{n,k},\widetilde B_{n,k^\prime})| = \mathcal O(c_n^{-\eta})
\]
which tends to 0 as $n\to \infty$. The first double sum in~\eqref{eq:covUn} is upper-bounded by $3^d \kappa \Var(\widetilde B_{n,0})$ and
\[
  \Var(\widetilde B_{n,0}) = \P(\Me_Z\leq Z_{n,0} \leq \Me_Z + y\sqrt{c_n/k_n}) \big(1-\P(\Me_Z\leq Z_{n,0} \leq \Me_Z + y\sqrt{c_n/k_n})\big).
\]
By Assumption (i)-(ii), $\Me_Z=\lambda c_n +o(\sqrt{c_n})$ and $\Me_Z+y\sqrt{c_n/k_n} = \lambda c_n+o(\sqrt{c_n})$ for every $y\in \R$. So we can apply~\eqref{eq:convtn} which leads to $\P(Z_{n,0} \geq \Me_Z)\to 1/2$, $\P(Z_{n,0}\leq \Me_Z+y\sqrt{c_n/k_n})\to 1/2$ and finally to $\Var(\widetilde B_{n,0})\to 0$ and  $\widetilde B_n\to 0$ in probability as $n\to \infty$.\\

Now we can  conclude. For all $\varepsilon>0$, there exists $n_0(\varepsilon)$ such that for
all $n\geq n_0(\varepsilon)$, $y_n\leq y + \varepsilon/2$.
Therefore for $n\geq n_0(\varepsilon)$
\begin{align*}
\P(A_n\leq y ,\; B_n(\Me_Z)\geq y&+\varepsilon)=\P(B_n(\Me_Z+y \sqrt{c_n/k_n}) \leq y_n,\; B_n(\Me_Z) \geq y+\varepsilon) \\
&\leq  \P(B_n(\Me_Z+ y \sqrt{c_n/k_n})\leq y\,{+}\,\varepsilon/2, B_n(\Me_Z)\geq y+\varepsilon) \\
 &\leq\P \left( \left| B_n(\Me_Z+y\sqrt{c_n/k_n} ) -B_n(\Me_Z)\right|\geq \varepsilon/2\right)  \\
 &\leq \P(|\widetilde B_n| \geq \varepsilon/2)
\end{align*}
which tends to 0 as $n\to \infty$ and (II) is proved.\\

\noindent(b) It is sufficient to combine Theorem~\ref{thm:cltF} (b) and Theorem~\ref{thm:bahadur} (a). From Slutsky's Lemma and by Assumptions (ii-2)-(ii-3), the following convergence in distribution holds as $n\to \infty$
\[
  \sqrt{k_n/c_n}  s_n\left( \Mee(\bZ) -\Me_Z\right) \to \mathcal N(0,1/4)
\]
where $s_n=\sqrt{c_n}\P(N(C_{n,0})=\lfloor \lambda c_n\rfloor)$. Since $\Mee(\bZ)=c_n\Mee(\check\bZ)$, $\Me_{ Z}=c_n\Me_{\check Z}$ and $|W_n|=k_n c_n$, this can be rewritten as 
\[
  |W_n|^{1/2} s_n \left( \Mee (\check\bZ) -\Me_{\check Z}\right) \to \mathcal N(0,1/4).
\]
From \eqref{eq:modifiedEst} and by Assumptions (i)-(ii), $\Me_{\check Z}=\lambda + \mathcal O(c_n^{\ell-1})$ and 
$\sqrt{k_n c_n} c_n^{\ell-1}\to 0$ as $n\to \infty$. Hence, a last application of Slutsky's Lemma concludes the proof.
\end{proof}

\section{Additional comments} \label{app:supplementary}


\subsection{The way of jittering a sample of counts}

We could think of slightly  generalizing~\eqref{eq:Znk}  and introduce the jittering effect as
\begin{equation*}\label{eq:Znkphi}
  Z_{n,k}  = N({C_{n,k}}) + \varphi^{-1}(U_k ) 
\end{equation*}
for any $k\in \mathcal K_n$, 
where $\varphi:[0,1]\to [0,1]$ is a continuously differentiable increasing function. 
The cumulative distribution function of $Z$ would be in that case
\[
  F_{Z}(t) = P(N(C_{n,0})\leq \lfloor t \rfloor -1) + P(N(C_{n,0})=\lfloor t\rfloor ) \,\varphi(t-\lfloor t \rfloor)
\]
and for any $t \notin \N$, $Z$ would admit a density $f_Z$ at $t$ given by
\begin{equation*}\label{eq:fZphi}
  f_Z(t) = P(N(C_{n,0})=\lfloor t\rfloor ) \, \varphi^\prime(t-\lfloor t \rfloor).
\end{equation*}
When $t\in \N$, since $(F_Z(t+h)-F_Z(t))/h$ tends to $\P(N(C_{n,0})=\lfloor t\rfloor)\varphi^\prime(0)$ when $h\to 0^+$ and to $\P(N(C_{n,0})=\lfloor t\rfloor)\varphi^\prime(1)$ when $h\to 0^-$, $Z$ would also admit a density at $t$ if we add the condition $\varphi^\prime(0)=\varphi^\prime(1)$. However, our Theorem~\ref{thm:bahadur} requires  another assumption. Namely, we need to assume that for any $t_n=\lambda c_n + \mathcal{O}(\sqrt{c_n/k_n})$, 
${f_Z(t_n)}/{f_Z(\lambda c_n)}$ tends to $1$. To this end, we would have to combine  Assumption (ii-2), with an assumption like $\inf_{t} \varphi^\prime(t)=\sup_t \varphi^\prime(t)$. This explains why we focused on the case $\varphi(t)=t$ in Section~\ref{sec:estimators} and in the presentation of our asymptotic results in Section~\ref{sec:asymp}.

\subsection{Rule of thumb under the Poisson case}

In this section, we want to examine the value of the true median of $Z$ under the Poisson case. Even if this is useless we also had a look at different functions $\varphi$. Figure~\ref{fig:medianPoisson} presents the true median of $\Pi$ and $Z=\Pi+\varphi^{-1}(U)$ where $\Pi$ follows a Poisson distribution with mean $\nu$ and where $U$ is a uniform random variable on $[0,1]$. We considered the cases $\varphi(t)=\sqrt{t},t,t^2$ and examined the true median minus $\nu$ in terms of $\nu$. First, we recover a result obtained by \cite{adell:jodra:05}: when $\nu$ is an integer, the median of $\Pi$ equals $\nu$ and for other values of $\nu$, it lies in the interval $[\nu-\log(2),\nu+1/3]$. It is worth   observing that the choice $\varphi(t)=t$ leads us to conjecture that when $\nu$ is large $\Me_{Z}$ is very close to $\nu+1/3$.

So, we could use the rule of thumb derived under the Poisson case and  modify the jittered estimator~\eqref{eq:lambdaJ} as follows
\begin{equation}
  \label{eq:modifiedEst}
  \widehat \lambda^{J,2} = \widehat \lambda^{J} -\frac1{3c_n}=\Me_{\check \bZ} -\frac1{3c_n}.
\end{equation}
Since $|W_n|^{1/2}/c_n=\sqrt{k_n/c_n}\to 0$ by Assumption~(i), this produces no differences asymptotically: $\widehat\lambda^{J,2}$ has the same behaviour as $\widehat\lambda^J$ and satisfies  the   central limit theorem given by~\eqref{eq:convEst} or Corollary~\ref{cor:lambdaEstCox}. We compared $\widehat\lambda^J$ and $\widehat\lambda^{J,2}$  in the framework of the simulation study presented in Section~\ref{sec:sim}. The evidence of better empirical results was unclear which explains why we did not present $\widehat\lambda^{J,2}$ before and kept $\widehat \lambda^J$ in the simulation study.

\begin{figure}[H]
\centering
\includegraphics[scale=.7]{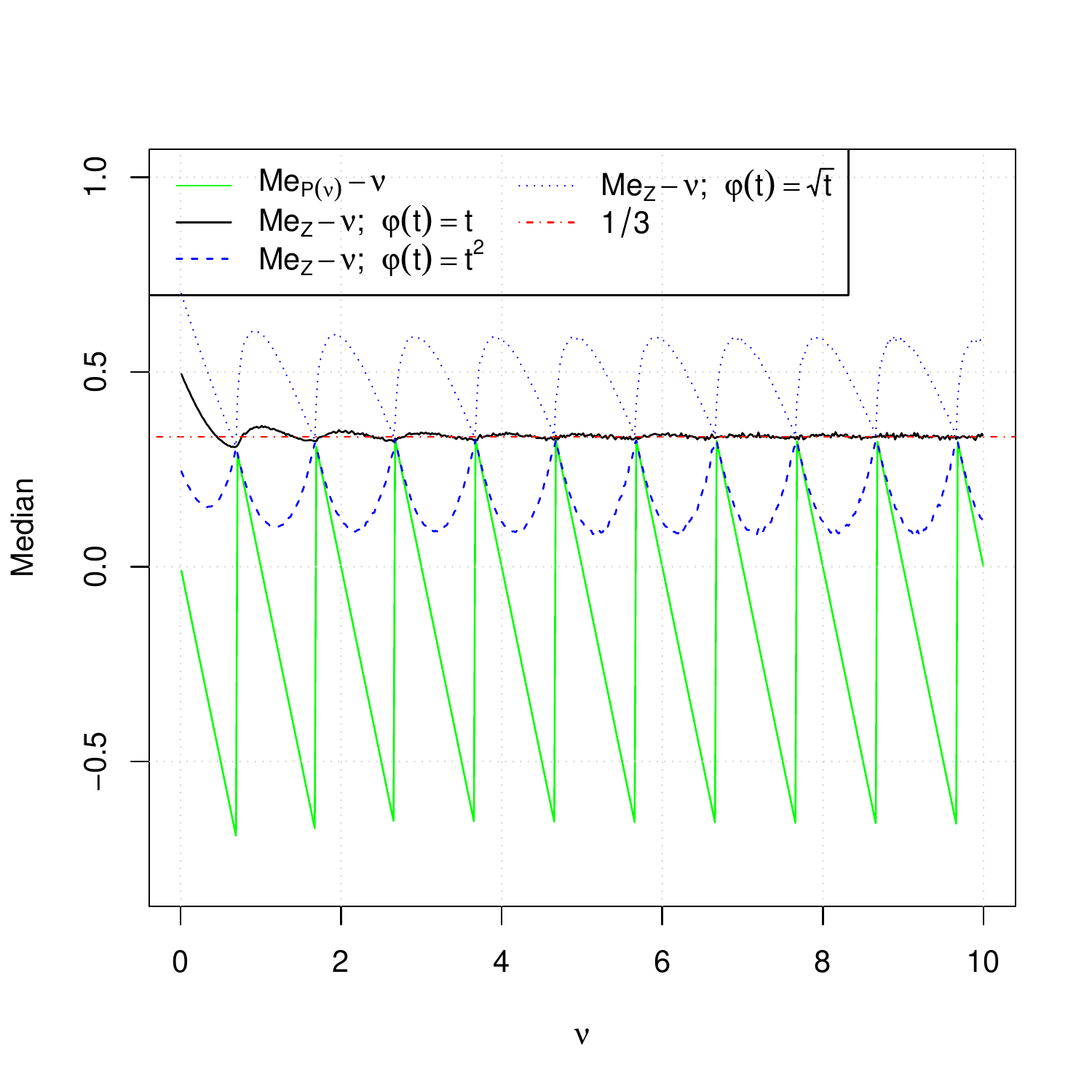}
\caption{\label{fig:medianPoisson} Sample medians based on $10^6$ replications of $\Pi$ or $\Pi+\varphi^{-1}(U)$ random variables where $\Pi$ (resp. $U$) follows a Poisson distribution with mean $\nu$ (resp. uniform distribution on [0,1]) and where $\varphi=(t)=t,t^2,\sqrt(t)$.}
\end{figure}

\bibliographystyle{plainnat}

\bibliography{intensity}

\end{document}